\documentclass[12pt]{amsart}
\usepackage{amssymb,amscd,amsmath,amsthm,color}
\usepackage{calc}
\usepackage{amsrefs}
\usepackage{color}
\usepackage[T1]{fontenc}
\usepackage{mathtools}
\usepackage[normalem]{ulem}
\usepackage{tikz}
\usetikzlibrary{matrix,arrows,decorations.pathmorphing}
\usepackage{setspace}
\usepackage{verbatim}
\usepackage{mathrsfs}
\usepackage{lineno}

\definecolor{bettergreen}{rgb}{0,.7,0}

\makeatletter
\newcommand*{\rom}[1]{\text{\expandafter\@slowromancap\romannumeral #1@}}
\makeatother

\newcommand{\op}[1]{\operatorname{#1}}
\newcommand{\ring}[1]{{\mathbb #1}}
\newcommand{\cal}[1]{\mathcal{#1}}
\newcommand{\locus}[1]{\op{Iv}(#1)}
\def\NF{\op{NF}}
\def\UF{\op{UF}}
\def\VF{{\op{VF}}}

\def\R{\cal{R}}
\def\Y{\Upsilon}
\def\s{{\mathfrak{f}}}

\def\bf{\mathbf f}

\newcommand{\cF}{\mathcal{F}}

\newcommand{\fg}{\mathfrak{g}}
\newcommand{\fb}{\mathfrak{b}}
\newcommand{\fn}{\mathfrak{n}}
\newcommand{\fc}{\mathfrak{c}}

\newcommand{\fh}{\mathfrak{h}}

\newcommand{\reg}{\mathrm{rss}}

\newcommand\ord{\mathrm{ord}}
\newcommand\ac{\overline{\mathrm{ac}}}

\newcommand{\Loc}{\mathrm{Loc}}

\theoremstyle{plain}
\newtheorem{thm}{Theorem}
\newtheorem{theorem}[thm]{Theorem}
\newtheorem{lem}[thm]{Lemma}
\newtheorem{cor}[thm]{Corollary}

\theoremstyle{definition}
\newtheorem{rem}[thm]{Remark}

\title{Endoscopic transfer of orbital integrals in large residual characteristic}

\author{Julia Gordon and Thomas Hales}

\begin{document}

\begin{abstract} This article constructs Shalika germs in the context
  of motivic integration, both for ordinary orbital integrals and
  $\kappa$-orbital integrals.  Based on transfer principles in motivic
  integration and on Waldspurger's endoscopic transfer of smooth
  functions in characteristic zero, we deduce the endoscopic
  transfer of smooth functions in sufficiently large residual
  characteristic.
\end{abstract}

\maketitle

We dedicate this article to the memory of Jun-Ichi Igusa.  The second
author wishes to acknowledge the deep and lasting influence that
Igusa's research has had on his work, starting with his work as a
graduate student that used Igusa theory to study the Shalika germs of
orbital integrals, and continuing today with themes in motivic
integration that have been inspired by the Igusa zeta function.

\bigskip

This article establishes the endoscopic matching of smooth functions
in sufficiently large residual characteristic.  The main conclusions
are based on four fundamental results: Langlands-Shelstad descent for
transfer factors~\cite{LSxf}, Ng\^o's proof of the fundamental
lemma~\cite{ngo2010lemme}, Waldspurger's proof that the fundamental
lemma implies endoscopic matching of smooth functions in
characteristic zero~\cite{waldspurger1997lemme}, and the
Cluckers-Loeser version of motivic integration~\cite{CL}, including
transfer principles for deducing results for one nonarchimedean field
from another nonarchimedean field with the same residue
field~\cite{CLe}.  We use recent extensions of the transfer principle
to transfer linear dependencies from one field to another \cite{CGH2}.

We note that the term transfer is used with two separate meanings in
this article.  {\it Endoscopic matching}\footnote{It is called
  endoscopic transfer in the literature, title, and abstract, but we
  prefer to call it endoscopic matching in the body of the article
  because of the other uses of the word transfer. We avoid the awkward
  but apt phrase ``transfer of transfer.''}  refers to the matching of
$\kappa$-orbital integrals on a reductive group with stable orbital
integrals on an endoscopic group, in a form made precise by the
Langlands-Shelstad transfer factor.  On the other hand, {\it transfer
  principles} refer to the transfer of first-order statements or
properties of constructible functions from one nonarchimedean field to
another nonarchimedean field with the same residual characteristic.  In
this article, we will refer to endoscopic matching, transfer factors,
and transfer principles.

In our main results, the constraints on the size of the residual
characteristic are not effective.  This means that our results are not
known to apply to any particular nonarchimedean field of positive
characteristic. This seems to be a serious limitation of our methods.
Nonetheless, we hope that our results about the constructibility of
Shalika germs can serve as a further illustration of the close
connection between harmonic analysis of $p$-adic groups and motivic
integration.

\subsection{Acknowledgement} J.G. is grateful to Radhika Ganapathy for
suggesting this project and helpful conversations at the beginning; we
also thank MSRI, where this collaboration began.

\section{Statement of results}

This introductory section surveys the main results of this article.
Unexplained terminology and notation are explained in the main body of
the article.  

We use the concepts of {\it definable sets} and {\it constructible
  motivic functions} (or constructible functions for short) from
\cite{CL}.  All constructible functions in this article are understood
in this sense. In fact, in \cite{CL} a closely related notion of
\emph{definable subassignment} is used instead of our notion of a
\emph{definable set}, but this distinction does not affect the
contents of this article.

In this article, by nonarchimedean field we mean a non-discrete
locally compact nonarchimedean valued field; that is, a field
isomorphic to a finite extension of $\ring{Q}_p$ or $\ring{F}_p(\!(t)\!)$
for some prime $p$.  Let $\op{Loc}_m$ denote the set of all
nonarchimedean fields whose residual characteristic is at least
$m\in \ring{N}=\{0,1,2,\ldots\}$.  To avoid set-theoretic issues, we
assume that (the carrier of) each field is a subset of some fixed
cardinal number (that is sufficiently large to obtain every such field
up to isomorphism).  If $S$ is a definable set, and $F\in \op{Loc}_m$,
we write $S(F)$ for the interpretation of $S$ in $F$.  If $f$ is a
constructible function on $S$, we write $f_F$ for the corresponding
function $f_F:S(F)\to \ring{C}$.

Let $G_{/Z}$ be a definable reductive group over a definable cocycle
space $Z$ in the sense of \cite{CGH} with a given set of fixed choices
enumerated in Section~\ref{sec:fixed}.  For each $k\in\ring{N}$, there
is a definable set $\NF^k_G$ representing $k$-tuples of Barbasch-Moy
pairs $\Y=(N,\s)$, in the sense of Section~\ref{sec:nilpotent}, with
pairwise non-conjugate nilpotent elements $N$ in the Lie algebra $\fg$
of $G$.  Here, $\s\in\cF$, where $\cF$ is a parahoric indexing set.

For $F\in \op{Loc}_m$ and $z\in Z(F)$, we have $F$-points $G_z(F)$ of
a connected reductive group over $F$, obtained by twisting the split
group by the cocycle $z$. We also have $F$-points $\fg_z(F)$ of a Lie
algebra. There is a definable set $\fg^\reg$ of regular semisimple
elements in $\fg$.  

The Shalika germs of orbital integrals were first constructed in
\cite{shalika1972theorem}.  Our first theorem establishes the
existence of constructible motivic functions that represent Shalika
germs.  The following theorem is proved in Section~\ref{sec:msg}.

\begin{theorem} [Lie algebra Shalika germs]\label{thm:lie-shalika} 
  For each $n\in \ring{N}$, there exists an $n$-tuple $\Gamma$ of
  constructible functions with domain $\fg^\reg\times_Z \NF^n_G$ and a
  constructible function $d_n$ with domain $\NF^n_G$.  There exists
  $m\in \ring{N}$, such that for every field $F\in \op{Loc}_{m}$, and
  for all $z\in Z(F)$, if $k=k_{F,z}$ is the number of nilpotent
  orbits in $\fg_z(F)$, then $d_{k,F}$ does not vanish anywhere on the
  nonempty set $\NF^{k}_{G_z}(F)$.  Moreover, for every $\Y\in
  \NF^{k}_{G_z}(F)$, the function $X\mapsto
  \Gamma(X,\Y)_{i,F}/d_{k,F}(\Y)$ is the Shalika germ (for some
  normalization of measures) at $X\in \fg^\reg_z(F)$ on the orbit of
  $N_i$, where $N_i$ is the first component of $\Y_i=(N_i,\s_i)$, for
  $i=1,\ldots,k$.
\end{theorem}

Although the functions are defined on the full set of regular elements
on the Lie algebra $\fg$, it is only their behavior in a small
neighborhood of $0$ that matters.  We make no claims about the their
behavior far from $0$.

The first components of $\Y_i=(N_i,\s_i)$, for $i=1,\ldots,k$, give a
complete non-redundant enumeration of nilpotent orbits of $G_z$.  The
theorem realizes all Shalika germs as constructible functions.  In a
sense that can be made precise (the residual characteristic must be
large with respect to fixed choices of the field-independent data such
as the root data), every large residual characteristic reductive group
can be represented as fiber $G_z$ in a definable reductive group $G_{/Z}$,
so that the theorem is general in large residual characteristic.

We obtain a similar representation as constructible functions for
Shalika germs in the group in terms of unipotent conjugacy classes.
For each $n$, there is a definable set $\UF^n_G$ representing
$n$-tuples of unipotent Barbasch-Moy pairs in $G$ whose tuple of first
coordinates are pairwise non-conjugate (Section~\ref{sec:nilpotent}).
The following theorem is proved in Section~\ref{sec:adapt}.

\begin{theorem} If we replace $\NF^n_G$ (resp. $\NF^{k}_{G_z}(F)$,
  $\fg^\reg$) with $\UF^n_{G}$ (resp. $\UF^{k}_{G_z}(F)$, $G^\reg$)
  in Theorem~\ref{thm:lie-shalika}, the same statement holds in the
  group.
\end{theorem}

Our next theorem gives a representation of the Langlands-Shelstad
transfer factor as a constructible function.  This result was
previously known for unramified groups \cite{CHL}.  The proof appears
in Section~\ref{sec:lsxfer}. 

\begin{theorem}[Constructibility of Lie algebra transfer
  factors]\label{thm:xfer-factor} 
  Let $(G,H)_{/Z}$ be a definable reductive group and associated
  endoscopic group over a cocycle space $Z$, with Lie algebras $\fg$
  and $\fh$.  There exists a constructible function $\Delta$ on
\begin{equation}\label{eqn:delta-domain}
V := \fh^{G-\reg}\times_Z
  \fg^\reg\times_Z
\fh^{G-\reg}\times_Z
  \fg^\reg
\end{equation}
and a natural number $m\in \ring{N}$ such that for every field $F\in
\op{Loc}_{m}$, every $z\in Z(F)$, and for every 
\[
(X_H,X_G,\bar X_H,\bar X_G)\in
V_z(F),\]  
the Lie algebra Langlands-Shelstad transfer factor is
$\Delta_F(X_H,X_G,\bar X_H,\bar X_G)$.
\end{theorem}

Here, $\fh^{G-\reg}$ denotes the definable set of $G$-regular
semisimple elements in $\fh$.  We are unable to obtain the
Langlands-Shelstad transfer factor on the full group, because there is
currently no good theory of multiplicative characters for motivic
integration. The following theorem is proved in
Section~\ref{sec:adapt}, by adapting the Lie algebra proof.

\begin{theorem}[Constructibility of group transfer factors]
\label{thm:gtf} 
Let
  $(G,H)_{/Z}$ be a definable reductive group and associated
  endoscopic group over a cocycle space $Z$.  
  There exists a constructible
  function $\Delta$ on the set of $G$-regular elements in
   a definable neighborhood $V$ of $1$ in
\[
H\times_Z  G\times_Z H\times_Z   G
\] 
and a natural number $m\in \ring{N}$ such that for every field $F\in
\op{Loc}_{m}$, every $z\in Z(F)$, and for every 
\[
(\gamma_H,\gamma_G,\bar \gamma_H,\bar \gamma_G)\in V^{G-\reg}_z(F),\] the
Langlands-Shelstad transfer factor for the group $G_z(F)$ restricted to
$V_z(F)$ is $\Delta_F(\gamma_H,\gamma_G,\bar \gamma_H,\bar
\gamma_G)$.
\end{theorem}

Based on the the constructibility of Shalika germs and the transfer
factor, we obtain a constructible function representing the Shalika
germs of $\kappa$-orbital integrals (Equation~\ref{eqn:kappa}).  As a
special case, we obtain the constructibility of stable Shalika germs
(Equation~\ref{eqn:stable}).  These constructions work both in the Lie
algebra and in the group.  The following theorem is proved in
Section~\ref{sec:adapt}.  An alternative version of the theorem for
Lie algebras is proved in Section~\ref{sec:emsg}.

\begin{theorem}[local endoscopic matching of orbital integrals]\label{thm:local}
  Let $(G,H)_{/Z}$ be a definable reductive group $G$ with definable
  endoscopic group $H$ over a cocycle space $Z$.  There exists a
  natural number $m\in \ring{N}$ such that for all $F\in
  \op{Loc}_{m}$, and all $z\in Z(F)$, the local endoscopic matching of
  orbital integrals holds for $(G_z,H_z)$.  That is, for all $f\in
  C_c^\infty(G_z(F))$, there exists a function $f^H\in C_c^\infty(H_z(F))$ for
  which the $\kappa$-orbital integrals of $f$ (with the usual
  Langlands-Shelstad transfer factor) are equal to the stable orbital
  integrals of $f^H$ near $1$. More precisely,
\[
SO(\gamma_H,f^H) = \sum_{\gamma_G} \Delta_F(\gamma_H,\gamma_G,\bar
\gamma_H,\bar \gamma_G) O(\gamma_G,f),
\]
for all $\gamma_H\in H^{G-\reg}_z(F)$ near $1$.
\end{theorem}

Our wording in Theorem~\ref{thm:local} is based on definition of {\it
  local $\Delta$-matching\footnote{It is called local $\Delta$-transfer in
    \cite{LSd}.} at the identity} in \cite{LSd}.  It can be
expressed equivalently in terms of Shalika germs.  The wording in
Theorem~\ref{thm:xfer} is based on the definition of {\it
  $\Delta$-matching} in \cite{LSd}.  Recall that $\Delta$-matching
(and the definition of the Langlands-Shelstad transfer factor away
from $1$) requires a choice of central extension $\tilde H$ of the
endoscopic group $H$ and a character $\theta$ on the center of
$\tilde H(F)$ \cite{LSxf}*{\S 4.4}.  The following theorem  is our main
result.  The proof appears in Section~\ref{sec:xfer}.

\begin{theorem}[endoscopic matching of orbital integrals]\label{thm:xfer}
  Let $(G,H)_{/Z}$ be a definable reductive group $G$ with definable
  endoscopic group $H$ over a cocycle space $Z$.  There exists a
  natural number $m\in \ring{N}$ such that for all $F\in
  \op{Loc}_{m}$, and all $z\in Z(F)$, the endoscopic matching of
  orbital integrals holds for $(G_z,H_z)$.  That is, for all $f\in
  C_c^\infty(G_z(F))$, there exists a $f^{\tilde H}\in
  C_c^\infty(\tilde H_z(F),\theta)$ for which the $\kappa$-orbital
  integrals of $f$ (with the usual Langlands-Shelstad transfer factor)
  are equal to the stable orbital integrals of $f^{\tilde H}$. More
  precisely,
\[
SO(\gamma_{\tilde H},f^{\tilde H}) = \sum_{\gamma_G}
\Delta_F(\gamma_{\tilde H},\gamma_G,\bar
\gamma_{\tilde H},\bar \gamma_G) O(\gamma_G,f),
\]
for all $\gamma_{\tilde H}\in {\tilde H}^{G-\reg}_z(F)$.
\end{theorem}

In a sense that can be made precise (the residual characteristic must
be large with respect to fixed choices of the field-independent data
such as the root data), every large residual characteristic reductive
group with endoscopic group can be represented as a fiber $(G_z,H_z)$
of a definable pair $(G,H)_{/Z}$, so that the theorem is comprehensive in
large residual characteristic.

We prove other theorems that we do not state in this introduction that
may be of interest.  We give a transfer principle for asymptotic
expansions (Theorem~\ref{thm:asym}) and prove a uniformity result for
the asymptotics as the nonarchimedean field varies
(Theorem~\ref{thm:uniform}).  We give a classification of definable
reductive groups in terms of fixed data
(Section~\ref{sec:classification}).  This classification departs in
noteworthy ways from the classification of reductive groups over
nonarchimedean fields.  This leads to the conclusion, for example,
that nonisomorphic unitary groups of the same rank over a
nonarchimedean field can have Shalika germs given by the same formula.
A final section lists some open problems
(Section~\ref{sec:open-problems}).

\section{Review of motivic integration in representation theory}

In this section we review basic constructs of motivic integration,
applied to representation theory and harmonic analysis, as developed
in \cite{CHL} and \cite{CGH}.  We assume familiarity with definable
sets and constructible functions, as developed in \cite{CL}.

We work in the Denef-Pas language, a three-sorted first-order
language.  The three sorts are called the valued field sort, the
residue field sort, and the value group sort.  The valued field sort
and residue field sort both contain the first-order language of rings,
and the value group sort contains the first-order language of ordered
groups.  There is function symbol $\op{ord}$ from the valued field
sort to the value group sort, and function symbol $\ac$ from the
valued field sort to the residue field sort, which are interpreted as
the valuation map and angular component map respectively.

By a {\it fixed choice}, we mean a fixed set that does not depend
on the the Denef-Pas language or its variables in any way.  Fixed
choices are assumed to be made at the outermost level, and
will sometimes be dropped from the notation.  Various formulas
in the Denef-Pas language will be constructed from fixed choices.

A free parameter (or simply parameter) refers to a collection of free
variables of the same sort in a formula in the Denef-Pas language,
ranging over a definable set.  A bound parameter is similar, except
that the variables are all bound by a contiguous block of existential
or a contiguous block of universal quantifiers.  By a parameter in
$\VF$ (for valued field), we mean a variable of valued field sort.

\subsection{Fixed choices}\label{sec:fixed}

We let $G^{**}$ be a split connected reductive group over $\ring{Q}$,
with Borel $B^{**}$ and Cartan $T^{**}\subset B^{**} \subset G^{**}$
all over $\ring{Q}$.  We let $(X^*,X_*,\Phi,\Phi^\vee)$ (characters,
cocharacters, roots, and coroots) be the root
datum for $G^{**}$,
all with respect to
$(B^{**},T^{**})$.  

We let $\Sigma$ be a (large) abstract finite group (that plays the
role of a Galois group), with a fixed enumeration
$1=\sigma_1,\ldots,\sigma_n$ of its elements, where $n =
\op{card}{\Sigma}$.
We assume that $\Sigma$ comes with a fixed short exact sequence
\[
1 \to \Sigma^t \to \Sigma \to \Sigma^{unr} \to 1,
\]
where $\Sigma^t$ (called the tame inertia) and $\Sigma^{unr}$ (called
the unramified quotient) are both assumed to be cyclic.  We fix a
generator $\op{qFr}$ of $\Sigma^{unr}$, called the quasi-Frobenius
element.

We choose an action of $\Sigma$ on the root datum stabilizing the set
of simple roots.  Specifically, we fix a homomorphism
\[
\rho_G:\Sigma\to \op{Out}(G^{**},B^{**},T^{**},\{X_\alpha\}),
\]
from $\Sigma$ to the group of automorphisms of $G^{**}$ fixing a
pinning, which we identify with the group of outer automorphisms of
$G^{**}$.  We use the action $\rho_G$ on the root datum to form the
$L$-group
\[
{}^LG = \hat G \rtimes \Sigma,
\]
where $\hat G$ is the complex Langlands dual of $G^{**}$.
We fix $(\hat T,\hat B)$ dual to $(T^{**},B^{**})$.    We assume
that the action of $\Sigma$ on $\hat G$ preserves a pinning
$(\hat T,\hat B,\{\hat X_\alpha\})$ of $\hat G$.  

We choose a semisimple element $\kappa\in \hat T^\Sigma$, with
connected centralizer $\hat H = C_{\hat G}(\kappa)^0$.  We let
$(X^*,X_*,\Phi_H,\Phi_H^\vee)$ be the root datum dual to that
of $\hat H$.  We choose as part of the endoscopic data,
a homomorphism $\rho_H:\Sigma \to C_{{}^LG}(\kappa)/\hat T$,
through which we construct
\[
{}^LH = \hat H \rtimes \Sigma.
\]
We assume that $\Sigma$ acts on $\hat H$ through automorphisms that
preserve some pinning of $\hat H$.  This choice induces an action of
$\Sigma$ on the root datum (giving two actions, $\rho_G$ and $\rho_H$,
of $\Sigma$ on $X^*$: one from $G$ and one from the endoscopic data).
Temporarily disregarding $\Sigma$, we let $H^{**}$ be the split group
over $\ring{Q}$ with root datum $(X^*,X_*,\Phi_H,\Phi_H^\vee)$.  We
fix a linear algebraic subgroup $\op{Aut}_0(G^{**})$ with finitely many
connected components of $\op{Aut}(G^{**})$ that contains the full
group of inner automorphisms and the image of $\Sigma$.

To treat the theory of reductive groups in a definable context, we fix
faithful rational representations of all split reductive groups (over
$\ring{Q}$) of interest: $G^{**}$, $G^{**}_{\text{sc}}$, $G^{**}_{\text{der}}$,
$G^{**}_{\text{adj}}$,
$\op{Aut}_0(G^{**})$, and even $(G^{**}_{\text{sc}}\times
G^{**}_{\text{sc}})/\op{Zent}_{\text{sc}}$, with diagonal embedding of the center
$\op{Zent}_{\text{sc}}$ of the simply connected cover of the derived group of
$G^{**}$.  We do the same for $H^{**}$ and its affiliated groups.  We
fix the obvious morphisms ($G^{**}\to G^{**}_{\text{adj}}$, and so forth)
between affiliated groups by explicit polynomial maps (expressed as
collections of polynomials with rational coefficients).  The rational
representations of the reductive groups give
representations of all of the associated Lie algebras.

The root datum and pinning for $G^{**}$ give a Dynkin diagram
$\op{Dyn}(\Phi^+(G^{**}))$ whose nodes are the simple roots of
$(G^{**},B^{**},T^{**})$.  The group $\Sigma$ acts on the Dynkin
diagram by permutation of the nodes.  Let $\tau$ be an index running
over the set of connected components of the Dynkin diagram.  For each
connected component $\op{Dyn}_\tau$ of the diagram, we consider the
subgroup $\Sigma_\tau$ stabilizing the component, and the kernel
$\Sigma^0_\tau\subset \Sigma_\tau$ of the action on the component.  We
let $\ell_\tau$ be the number of connected components in the $\Sigma$
orbit of $\tau$.  We let $d_\tau = [\Sigma_\tau:\Sigma_\tau^0]$ be the
index.  We add the prefix superscript $d_\tau$ to the symbol for the
connected Dynkin diagram in the usual way: ${}^1A_n$, ${}^2A_n$, and
so forth.

Following \cites{reeder2010torsion, Gross}, for each $\tau$, we
form an affine diagram ${}^e\R_\tau$ attached to the tame inertia
group $\Sigma^t_\tau = \Sigma^t\cap \Sigma_\tau$.  Its set of nodes is
$\{0\}\cup I_\tau$, where $I_\tau$ is the set of orbits of simple
roots of $\op{Dyn}_\tau$ under the action of $\Sigma^t_\tau$.  The
element $\{0\}$ represents an extended node.  The superscript prefix
$e$ is the order $[\Sigma^t_\tau:\Sigma^t_\tau\cap \Sigma_\tau^0]$ of
the group of tame inertial automorphisms of $\op{Dyn}_\tau$.  We drop
superscript $e$ from the notation when $e=1$.  The various connected
affine diagrams ${}^e\R = \cal{A}_n, {}^2\cal{A}_n, \cal{B}_n,
\cal{C}_n,\ldots, \cal{E}_8$ are listed in a table in \cite{Gross}.

We set $\Sigma^{unr}_\tau = \Sigma_\tau/\Sigma_\tau^0\Sigma^t_\tau$.
For each $\tau$, we choose an action
\[
\phi_\tau:\Sigma^{unr}_\tau\to \op{Aut}({}^e\R_\tau)
\]
of $\Sigma^{unr}_\tau$ on the affine diagram.  When $e\ne 1$, by
inspection of the various affine diagrams we see that the order of the
automorphism group is at most $2$, so that the choice of $\phi_\tau$
amounts to at most a binary choice of whether the action is trivial or
nontrivial.  When $e=1$, the situation is only slightly more
involved.

From each orbit of $\Sigma$ on the set of components of the Dynkin
diagram $\op{Dyn}(\Phi^+(G^{**}))$ we choose a representative
$\tau$.  Let $A = \{\tau\}_\tau$ be this set of representatives.  Let
$S$ (depending on all the choices above) be the product
\[
S = \prod_{\tau\in A} \op{node}({}^e\R_\tau)/\phi_\tau,
\]
of the $\phi_\tau$-orbits of nodes in each affine diagram.  Let $\cF
= \cF(G^{**},\Sigma,\rho_G,\phi_\tau)$ be the set of all
subsets $S'$ of $S$ such that the projection of $S'$ to each factor is
nonempty.  We call $\cF$ the parahoric indexing set.

The next subsection will give various parameter spaces that admit
interpretations in nonarchimedean structures.  When the fixed data
becomes associated with a reductive group $G$ over a nonarchimedean
field $F$, we either get nothing (if our data does not satisfy
required compatibility requirements) or we get the actual indexing set
for the parahoric subgroups in that reductive group.  This follows
directly from the explicit description of parahorics in \cite{Gross},
which is the starting point for our definition of $\cF$.  Parameters
$\s\in \cF$ may be identified with barycentric centers of facets in a
standard alcove in a suitable apartment of the Bruhat-Tits building.
Groups and algebras in the Moy-Prasad filtration may be associated
with these points in the building.  By \cite{CGH}, for each parameter
$\s\in \cF$, there are definable sets $G_{\s,0+}$, $\fg_{\s,0}$,
$\fg_{\s,0+}$ the pro-unipotent radical of the parahoric subgroup, a
parahoric subalgebra, and one of its subalgebras.

\subsection{Parameter spaces}

We recall our conventions for handling Galois cohomology and the
theory of reductive groups in the Denef-Pas language.  Further details
about these constructions can be found in \cite{CHL} and \cite{CGH}.
We describe various parameter spaces (or simply spaces for short) with
variables of the valued field sort.

\subsubsection{Field extensions}

The parameter space of field extensions\footnote{When we work with
  field extensions, we will write $E/\VF$, rather than the more
  cumbersome $\op{VE}/\VF$, etc.} $E/\VF$ of a fixed degree $n$ is defined
to be the parameter space of $(a_0,\ldots,a_{n-1})\in \VF^n$ (that is,
$n$ variables of the valued field sort) such that
\[
p_a(t) = t^n + a_{n-1} t^{n-1} + \cdots a_0
\]
is an irreducible polynomial.  Field arithmetic is expressed in terms
of operations on $\VF^n$ by means of the identifications
\[
E = \VF[t]/(p_a(t)) \simeq \VF^n.
\]

The space of automorphisms of $E/\VF$ consists of linear maps (brought
back to arithmetic on $\op{End}(E) = \op{End}{\VF^n} = \VF^{n^2}$)
respecting the field operations.  The space of Galois field extensions
$E/\VF$ of fixed degree $n$ is definable by the condition that there
exist $n$ distinct automorphisms of the field $E/\VF$.  There is a
space of Galois field extensions of fixed degree $n$, with an
enumeration of its automorphisms, given by tuples
\begin{equation}\label{eqn:enum-gal}
(E,\sigma_1,\sigma_2,\ldots,\sigma_n),
\end{equation}
where $E/\VF$ is a Galois field extension and $\sigma_i$ are the
distinct field automorphisms.  For fixed abstract group $\Sigma$ of
order $n$ with full enumeration $\sigma'_i\in \Sigma$, for
$i=1,\ldots,n$, there is a space of Galois field extensions with
Galois groups isomorphic to $\Sigma$ given by tuples in Formula
\ref{eqn:enum-gal} with the additional isomorphism requirement that
\[
\sigma_i\sigma_j = \sigma_k \Longleftrightarrow \sigma'_i\sigma'_j =
\sigma'_k.
\]

The space of unramified field extensions $E/\VF$ of a fixed degree $n$
is a definable subspace of the space of all field extensions.  There
is a space of pairs $(E,\op{qFr})$ where $E$ is an unramified extension
and $\op{qFr}$ is a fixed generator of the Galois group of $E/\VF$.  

There is a space of field extensions $K/E/\VF$, for fixed degrees
$E/\VF$ and $K/\VF$.  It is specified by irreducible polynomials
$p$ and $q$ for $E = \VF[t]/(p(t))$ and $K = \VF[t']/(q(t'))$, and an
element $t''\in \VF[t']/(q(t'))$ (the image of $t$ under $E\to K$) such
that $p(t'') = 0$.  There is a space of Galois field extensions
$K/E/\VF$ for fixed degrees for $E/\VF$ and $K/\VF$ with enumerations
of the Galois groups of $K/\VF$ and $K/E$, together with a table
describing the homomorphism $\op{Gal}(K/\VF)$ to $\op{Gal}(E/\VF)$.

If we have a short exact sequence of enumerated groups 
\[
1 \to \Sigma^t \to \Sigma \to \Sigma^{unr}\to 1,
\]
there is a space of Galois field extensions $K/E/\VF$ with enumerated
automorphisms $\sigma_1,\ldots,\sigma_n$, such that the automorphisms in
$\Sigma^t$ act trivially on $E$, $K/E$ is totally ramified with Galois
group $\Sigma^t$, $K/\VF$ has Galois group $\Sigma$, and $E/\VF$ is
unramified with Galois group $\Sigma^{unr}$.

\subsubsection{Galois cocycles}

The space of Galois cocycles is given as tuples
\[
(K,\sigma_1,\ldots,\sigma_n,a_1,\ldots,a_n),
\]
where $a_i \in K^k$ for some $k$.  Here $K/\VF$ is a Galois extension of
fixed degre $n$, with enumerated Galois group $\sigma_i$ subject to
the cocycle relations:
\[
\sigma_i \sigma_j 
  = \sigma_k \Longrightarrow a_i \sigma_i(a_j) 
  = a_k\quad\text{for all } i,j,k.
\]

For fixed choices of $G^{**}$, action of $\Sigma$ on the root datum,
there is a space $(K,\sigma_1,\ldots,a_1,\ldots,G^*)$ of
quasi-split reductive groups split over $K$ with enumerated
isomorphism between the Galois group of $K/\VF$ and $\Sigma$.  Here
the elements $a_i$ enumerate the cocycle in the outer
automorphism group (identified with automorphisms fixing a pinning
$(B^{**},T^{**},\{X_\alpha\})$) defining the quasi-split form.  The
pair $(B^{**},T^{**})$ gives a pair $(B^*,T^*)$ in the quasi-split
form.

$G$ is a definable family of reductive groups parametrized by a
definable cocycle space $Z$, which is constructed as in \cite{CGH}.
We include in the space $Z$ an explicit choice of inner twisting
\[
\psi:G
\times_\VF K \to G^*\times_\VF K
\] to the quasi-split inner form that agrees with an enumerated
cocycle $\sigma_i(\psi) \psi^{-1}$ with values in $G^*_{\text{adj}}(K)$, also
given as part of the data of $Z$.  In fact, we identify $G$ with $G^*$
over $K$, so that $\psi$ is the identity, but it is retained in
notation in the form of a cocycle $\sigma(\psi)\psi^{-1}$.  

Let $\fc_{\fg^{*}} = \fg/\!/G$ be Chevalley's adjoint quotient, which
we identify with the space of ``characteristic polynomials'' $\fg\to
\fc_{\fg^{*}}$.  The fiber in $\fg^\reg$ over each element of
$\fc_{\fg^{*}}$ is the definable stable conjugacy class with the given
characteristic polynomial.  Note that $\fc_{\fg^{*}}$ depends only on
the quasi-split form $G^*$. Similarly, we have a map $G\to \fc_{G^*}:=
(T^{**}/W)^*$, where $T^{**}$ is a maximally split Cartan subgroup of
the split form $G^{**}$ of $G$, and $W$ is the absolute Weyl group for
$T^{**}$ in $G^*$.  By $\fc_{G^*}$ we mean the twisted form of
$T^{**}/W$ obtained from the outer automorphisms of $T^{**}$ used to
define $G^*$.  That is, the Galois group $\op{Gal}(K/\VF)$ acts so
that $\VF$-points are
\[
t =  \phi_\sigma(\sigma(t))\mod W, \quad\forall \sigma\in\op{Gal}(K/\VF),
\]
where $t\mapsto \sigma(t)$ is standard action of the Galois
group on $T^{**}(K)$, and $\phi_\sigma$ is the outer automorphism
associated with $\sigma$.

If $F$ is a nonarchimedean field, and $z\in Z(F)$, we have the fiber
$G_z$.  It is a connected reductive linear algebraic group with Lie
algebra $\fg_z$.  The map $\fg_z\to \fc_{\fg^{*},z}$ is the usual map
from the Lie algebra to the space of characteristic polynomials.

\subsubsection{Endoscopic data}

We may similarly use the endoscopic data from our fixed choices to form
a space of quasi-split endoscopic data for $H$, again split by $K/\VF$.  
We let $\fg$ and $\fh$ be the Lie algebras of $G$ and $H$.

We
similarly construct a cocycle space for $H$ and take the fiber product
$Z$ of the cocycle spaces for $G$ and $H$, again denoting it by $Z$,
by abuse of notation.  We also include in the parameters of
the cocycle space $Z$ a large unramified extension and a choice
$\op{qFr}$ of generator.
We also
include in $Z$ a parameter $b$ that trivializes an invariant
differential form of top degree, as explained in Section~\ref{sec:volume}.

\subsubsection{Nilpotent elements}\label{sec:nilpotent}

The properties of nilpotent elements are well-known.  Here we
summarize the properties that we use.

\begin{theorem} 
  Let $G_{/Z}$ be a definable reductive group with Lie algebra
  $\fg_{/Z}$ over a cocycle space $Z$.  There exists $m\in\ring{N}$
  such that for all $F\in\op{Loc}_m$, and all $z\in Z(F)$, the
  following properties are equivalent properties of $N\in \fg_z(F)$:
\begin{enumerate}
\item $\rho(N)$ is nilpotent for {\it some} faithful rational
  representation $\rho$ of $\fg_z$;
\item $\rho(N)$ is nilpotent for {\it every} faithful rational
  representation $\rho$ of $\fg_z$;
\item $0$ lies in the Zariski closure of the adjoint orbit of $N$;
\item $0$ lies in the $p$-adic closure of the adjoint orbit of $N$;
\item there exists $\lambda\in X_*^F(G_z)$ such that $\lim_{t\to 0}
  \op{Ad}(\lambda(t))N = 0$; and
\item the image of $N$ under the morphism $\fg_z\mapsto \fc_{\fg^*_z}$
  is $0$.
\end{enumerate}
\end{theorem}

\begin{proof} All the implications mentioned in this proof should be
  understood as implications for the residual characteristic
  sufficiently large, with effective bound that depends only on the
  absolute root datum of $G$. Such bounds are stated in the sources we
  cite. The equivalence of the first two properties is in 
  \cite{humphreys1975linear}*{\S15.3}.  The implications (5)
  $\Leftrightarrow$ (4) $\Rightarrow$ (3) are in \cite{debacker:nilp},
  ~\cite{adler-debacker:bt-lie}*{2.5.1}.  The implications (2)
  $\Leftrightarrow$ (3) $\Rightarrow$ (5) are in \cite{mcninch}*{3.5},
  \cite{mcninch}*{4.1,Prop. 4 and Theorem 26}.

  To see the equivalence of (6) with say (2), we may work over an
  algebraically closed field. In particular, $\fg = \fg^{**}$.  Note
  that the image of $N$ is $0$ in $\fc_{\fg}$ if and only if the image
  of $N$ is $0$ in the torus $\fb/\fn$ (under conjugation to a Borel
  subalgebra), which holds if and only if the image of $N$ is in
  $\fn$, the nilradical of $\fb$.  This holds if and only if the
  semisimple part of $N$ is trivial, which is equivalent to (2).
\end{proof}

The last of the properties enumerated in the theorem is a definable
condition.  Thus, we have a definable set of all nilpotent elements in
$\fg$.  The theorem gives the compatibility of this definition with
notions of nilpotence in the various articles we cite.  The following
theorem appears in \cite{barbasch-moy}.

\begin{theorem}[Barbasch-Moy]\label{thm:bm} 
  Let $G_{/Z}$ and $\fg_{/Z}$ be as above.  There exists $m$ such that
  for all $F\in \op{Loc}_m$, $z\in Z(F)$, and for every nilpotent
  orbit in $\fg_z(F)$, there exists $\s\in \cF$ and a nilpotent
  element $N$ in the orbit such that
\begin{enumerate}
   \item $N\in \fg_{z,\s,0}$, and
   \item if $N'$ is nilpotent and $N'\in N + \fg_{z,\s,0+}$, then $N$
     lies in the $p$-adic closure of the orbit of $N'$.
\end{enumerate}
\end{theorem}

We say that $\Y=(N,\s)$ is a Barbasch-Moy pair, with $\s\in \cF$, if $N$
is nilpotent and it satisfies the properties of the theorem.  For each
$\s\in \cF$, there is a definable set consisting of all nilpotent
elements $N$ such that $\Y=(N,\s)$ is a
Barbasch-Moy pair.

\begin{cor}\label{thm:nilbound}  Let $G_{/Z}$ and $\fg_{/Z}$ be as above.  
There exists a constant $k$ such that for all $F\in \op{Loc}_m$, $z\in
Z(F)$, the number of nilpotent conjugacy classes in $\fg_z(F)$ is at
most $k$.
\end{cor}

\begin{proof} By the preceding theorem, for given $F$, the number of
  nilpotent classes is at most the sum of the numbers $k_\s$, where
  $k_\s$ is the number of nilpotent conjugacy classes in the finite
  reductive group $\fg_{z,\s,0}/\fg_{z,\s,0+}$.  Field-independent
  bounds on the number of nilpotent elements in a reductive
  group over a finite field are well-known \cite{carter1985finite}.
\end{proof}

We have corresponding versions of these theorems for the set of
unipotent elements.  We have a definable subset of $G$ consisting of
unipotent elements $u$, determined by the condition that the image of
$u$ under $G\to \fc_{G^*}$ is $1$.

\begin{theorem}\label{thm:bmg} 
  Let $G_{/Z}$ be as above.  There exists $m$ such that for all $F\in
  \op{Loc}_m$, $z\in Z(F)$, and for all unipotent conjugacy classes in
  $G_z(F)$, there exists $\s\in \cF$ and an element $u$ in the
  unipotent conjugacy class such that
\begin{enumerate}
   \item $u\in G_{z,\s,0}$, and
   \item if $u'$ is unipotent and $u'\in u G_{z,\s,0+}$, then $u$
     lies in the $p$-adic closure of the conjugacy class of $u'$.
\end{enumerate}
\end{theorem}

\begin{proof} First, we show that (1) and (2) are first-order
  statements in the Denef-Pas language.  The statement (2) is
  first-order because the set $G_{z,\s,0+}$ is definable, by
  \cite{CGH2}*{Lemma 3.4}. The set $G_{z,\s,0}$ appearing in (1) is
  not currently known to be definable in general, but by the same
  lemma, it is contained in a slightly larger definable set $G_\s^0$
  of Galois-fixed points of $G_z(K)_{\s, 0}$ (where $K$ is an
  extension that splits $G_z$).  However, if $u$ is a unipotent
  element contained in $G_\s^0$, then in fact $u$ is contained in
  $G_{z, \s, 0}$ by \cite{debacker:nilp}*{Lemma 4.5.1}.  Thus (2) is a
  first-order statement as well.
   
  Because these statements are first order in the Denef-Pas language,
  by a transfer principle, it is enough to prove the result when $F$
  has characteristic zero.  Now we work over $F$, with $z\in Z(F)$
  fixed, and write $\fg_{\s,r}$, $\fg_{\s,r+}$, $G_{\s,r}$,
  $G_{\s,r+}$ for the usual Moy-Prasad filtrations with $r\in\ring{R}$
  for the Lie algebra, and $r\in \ring{R}_{\ge 0}$ for the group.  In
  particular, $\fg_{\s,0}=\fg_{z,\s,0}(F)$, and so forth.

  When $F$ is characteristic zero, and $m$ is large enough, we have an
  exponential map defined on the set of all topologically nilpotent
  elements with the following properties.
  \begin{enumerate}
    \item $\exp(\fg_{\s,r}) = G_{\s,r}$, for all $r>0$.
    \item $\exp(\fg_{\s,0+}) = G_{\s,0+}$.
    \item $\exp$ restricts to a bijection between the set of
      nilpotent elements in $\fg_{\s,0}$ and the set of unipotent
      elements in $G_{\s,0}$.
   \item Let $N$ be a nilpotent element in $\fg_{\s,0}$.  Then the
     image of $\exp(N)$ in $G_{\s,0}/G_{\s,0+}$ is $\op{fexp}(\bar N)$,
     where $\bar N$ is the image of $N$ in $\fg_{\s,0}/\fg_{\s,0+}$ and
     $\op{fexp}$ is the finite field exponential from the nilpotent
     set of $\fg_{\s,0}/\fg_{\s,0+}$ to the unipotent set of
     $G_{\s,0}/G_{\s,0+}$. The map $\op{fexp}$ is injective on the
     nilpotent set.
   \item The exponential map from the nilpotent set to the unipotent
     set preserves orbits and the partial order given by orbit closure.
 \end{enumerate}
 These properties follow from \cite{kim:hecke}*{Prop. 3.1.1}, which
 shows that the exponential map defined by the usual series converges,
 and from the properties of mock exponential maps constructed by
 Adler \cite{adler:anisotropic}*{\S 1.5-1.6}.  When the exponential
 is defined, it is a mock exponential map.
   
  Pick $u$ in the given conjugacy class in such a way that $u =
  \exp(N)$, and $N$ has the properties in Theorem~\ref{thm:bm}, for
  some $\s\in \cF$.
  Then $u\in G_{z,\s,0}$.  Let $u' = \exp(N') \in u G_{z,\s,0+}$.
  Reducing modulo
  $G_{\s,0+}$ gives $\op{fexp}(\bar N') = \op{fexp}(\bar N)$. By
  injectivity, we have $\bar N' = \bar N$ and $N' \in N +
  \fg_{\s,0+}$.  By Theorem~\ref{thm:bm}, $N$ lies in the closure of the
  orbit of $N'$.  Exponentiating again, $u = \exp(N)$ lies in the
  closure of the orbit of $u' = \exp(N')$.
\end{proof}

We say that $\Y=(u,\s)$ is a unipotent Barbasch-Moy pair, with $\s\in
\cF$, if $u$ is unipotent and it satisfies the properties of the
theorem.  For each $\s\in \cF$, there is a definable set consisting of
all unipotent elements $u$ such that $\Y=(u,\s)$ is a unipotent
Barbasch-Moy pair.

\subsection{Volume forms}\label{sec:volume}

All integrals are to be computed with respect to invariant measures on
their respective orbits.  This means that the appropriate context for
motivic integration is integration with respect to volume forms, as
described in \cite[\S 8]{CL}.  Volume forms in the context of
reductive groups are explained in \cite{CGH}.  We follow those
sources.

We recall that each nilpotent orbit can be endowed with an invariant
motivic measure $d\mu^{\op{nil}}$ that comes from the the canonical
symplectic form on coadjoint orbits~\cite[Prop.~4.3]{CGH}.  

We obtain an invariant volume form on the unipotent set in the
group by exponentiating the volume form on the nilpotent cone.

The invariant volume forms on stable regular semisimple orbits in the
Lie algebra are constructed in \cite{CHL}.  We have the morphism
$\fg^\reg\to\fc_{\fg^*}$ that classifies the regular semisimple stable
orbits.  The Leray residue (in the sense of \cite{CL}) of the
canonical volume form on $\fg$ by the canonical volume form on
$\fc_{\fg^*}$ yields an invariant volume form $d\mu^\reg$ on each
fiber of the morphism, and hence an invariant volume form on stable
regular semisimple orbits with free parameter $X\in\fc_{\fg^*}$.  We
compute all stable orbital integrals on stable regular semisimple
orbits with respect to this family of volume forms.  By restriction to
each conjugacy class in the stable conjugacy class, we obtain an
invariant measure on each regular semisimple conjugacy class.

We construct an algebraic differential $d$-form on $G$, where $d$ is
the relative dimension of $G$ over $Z$.  We begin with the split case
(and $Z$ a singleton set).  If $G^{**}$ is split over $\op{VF}$, we
have a top invariant form on $G^{**}$ given on an open cell of
$G^{**}$ by
\begin{equation}\label{eqn:omegaK}
d^\times t\land dn\land dn',
\end{equation}
where $d^\times t$ is a top invariant form on a split torus $T^{**} \subset
B^{**}$, $dn$ a top invariant form on the unipotent radical
$R_u(B)^{**}$ of $B^{**}$, and $dn'$ on the unipotent radical of the
Borel subgroup opposite to $B^{**}$ along $T^{**}$.  We may pick root
vectors $X_\alpha$ for each root (as fixed choices over $\ring{Q}$).
Then by placing a total order on the positive roots, we may write
\[
n = \prod_{\alpha>0}\exp(x_\alpha X_\alpha),
\]
and $dn = \land_{\alpha>0} dx_\alpha$.  Similar considerations hold
for $dn'$.  It is known that this differential form extends to an
invariant differential form on all of $G^{**}$.

To go from the split case to the general case, 
let $G_{/Z}$ be a definable reductive group over a cocycle space $Z$.
Part of the data
for $Z$ is a field extension $K$ splitting $G$.  Let $G_K$
be the (definable) base change of $G$ from $\VF$ to $K$. We may view
$G$ as a subset of $G_K\subset \VF^{n}$, for some
$n\in\ring{N}$.

We may extend the top invariant form $\omega$ to $\omega_K$, from
$\VF$ to $K = \VF[t]/(p_a(t)) = \VF^k$, by taking each coordinate
$x_\alpha$, and so forth to be a monic polynomial of degree $k$ with
coefficients in $\VF$.  This does not change the degree $d$ of the
differential form, but expanding in $t$, it takes values in $\VF^k$.
The form $\omega_K$ is $G_K$-invariant.  It is also invariant by the
action of $G_{K,\text{adj}}$ on $G_K$ by conjugation.  We have a
morphism $G_K\to Z$, and $G$ is identified with the fixed point set of
the action of $\op{Gal}(K/\VF)$ on $G_K$.  Let $dz$ be any top form on
$Z$.

We describe the differential form when $G$ is quasi-split.  In this
case, the action of the Galois group fixes a pinning, the Galois group
permutes the root spaces and permutes the root vectors $X_\alpha$, up
to some structure constants.  From the expression for the differential
form in Formula \ref{eqn:omegaK}, we see that the Galois group
preserves the differential form $dz\land \omega_K$ up to a cocycle
with values in $K^\times$.

Even if the group is not quasi-split, the Galois group fixes the
pinning up to a cocycle $g_\sigma$ with values in $G_{K,\text{adj}}$,
and because the adjoint group action preserves $dz\land \omega_K$, we
see that the differential form is preserved by the Galois action, up
to a cocycle $a_\sigma$ with values in $K^\times$.  When we take
nonarchimedean structures, we may split this cocycle by Hilbert's
90th.  Working at the definable level, we introduce a new free
parameter $b$ into the cocycle space $Z$ taking values in $K^\times$
that is subject to the relation $a_\sigma = \sigma(b)^{-1} b$.

Set $\omega_0 = b\, dz\land\omega_K$.  It is Galois stable on
$G$, so on this definable set, the $d+\dim Z$-form takes values in $F$.
Also, since $\omega_K$ is invariant by $G_K$, this implies the
invariance of $\omega_0$ by $G$ (acting fiberwise over $Z$).  

We have a morphism $G^\reg\to \fc_{G^*}$ that classifies stable
regular semisimple conjugacy classes in the group.  We form the Leray
residue of the Haar volume form $\omega_0$ on $G$ with respect to the
canonical volume form on $\fc_{G^*}$.  This yields an invariant volume
form on each fiber of the morphism, and hence an invariant volume form
$d\mu^\reg$ on stable regular semisimple orbits in the group.

\subsection{Relation to valued fields}

To put the preceding definitions in context, we make a few comments
about reductive groups over nonarchimedean fields $F$ in large
residual characteristic.  In applications, we assume that the residual
characteristic is sufficiently large that all Galois groups that
appear are tame (that is, wild inertia is trivial).

\subsubsection{Multiplicative characters}

We recall that there is currently no good theory of multiplicative
characters for motivic integration.  In particular, since the
Langlands-Shelstad transfer factor makes use of multiplicative
characters in the form of $\chi$-data, when working motivically,
we restrict our attention to a neighborhood of the identity on the
group.  That is, we restrict to the kernels of the multiplicative
characters.

\subsubsection{Frobenius automorphism}

Although we may speak of unramified extensions and their Galois
groups, when working motivically, we have no way to single out the
canonical Frobenius generator of the Galois group of an unramified
extension.  Indeed, we do not have access the cardinality of the
residue field $q$, until we specialize to a particular nonarchimedean field
$F$.

Instead, in \cite{CHL} and here, we work with an arbitrary generator
$\op{qFr}$ of cyclic group $\Sigma^{unr}$.  In the Tate-Nakayama
isomorphism, the quasi-Frobenius element is used to identify
$\op{Gal}(E/F) = \ring{Z}/n\ring{Z}$, when $E/F$ is unramified of
degree $n$.  If we pick the ``wrong'' (i.e. non-Frobenius) generator
of $\op{Gal}(E/F)$, it has the effect of replacing the endoscopic
datum $s\in \hat T^\Sigma$ with another element $s^i$ with the same
centralizer, for some $i\in (\ring{Z}/n\ring{Z})^\times$.  This again
gives valid endoscopic data, and the arguments still work.

\section{Langlands-Shelstad transfer factor for Lie algebras}
\label{sec:lsxfer}

In this section, we prove the constructibility of the
Langlands-Shelstad transfer factor for Lie algebras as stated in
Theorem~\ref{thm:xfer-factor}.  Our definition will be carried out
entirely in terms of the Denef-Pas language and constructible
functions, without reference to a nonarchimedean field.  We note that
this theorem has already been established in the unramified
case~\cite{CHL}.  All the ideas of the proof are already present in
the unramified case.  The key observation is that the transfer factor
can be built up from Tate-Nakayama, field extensions, cocycles,
$SL(2)$-triples, conjugation, group operations, and inner twisting.
These concepts have all been developed in a definable context by
previous research \cite{CHL}, \cite{CGH}.  Readers who are willing to
accept without proof the constructibility of transfer factors may skip
ahead to Section~\ref{sec:tp}.

We make a fixed choice of a nonzero sufficiently divisible
$k\in\ring{Z}$.  We will give the definition of the transfer factor in
backwards order starting with the top-level description, and
successively expanding and refining the definitions until all unknown
terms have been specified.  The transfer factor $\Delta(X_H,X_G,\bar
X_H,\bar X_G)$ is the product of two constructible terms:
\[
\Delta_0(X_H,X_G,\bar
X_H,\bar X_G)\quad\text{and}\quad \ring{L}^{d(X_H,X_G)}.
\]
The parameters $X_H$ and $\bar X_H$ run over the definable set of
$G$-regular semisimple elements of the Lie algebra $\fh$.  The
parameters $X_G$ and $\bar X_G$ run over the definable set of regular
semisimple elements of $\fg$; that is, $(X_G,\bar X_G)\in
\fg^\reg\times_Z\fg^\reg$.  Recall that $\psi$ is the inner
isomorphism defined over an extension of $\VF$ between $G$ and $G^*$.
The transfer factor is defined to be $0$ unless $X_H$ and $X_G$
correspond under the definable condition requiring the image of
$\psi(X_G)$ in the Chevalley quotient $\fc_{\fg^*}$ to equal the image
of $X_H$ in $\fc_{\fh}$ under $\fc_{\fh}\to\fc_{\fg^*}$, and similarly
for $(\bar X_H,\bar X_G)$.  (Note that the endoscopic groups $H$ is
quasi-split by definition, so that $H=H^*$ and $\fh=\fh^*$.)  We now
assume that the parameters are restricted to the definable subsets
satisfying these constraints.  The factor $\ring{L}^{d(X_H,X_G)}$ is
the usual discriminant factor (called $\Delta_{\rom4}$ by Langlands
and Shelstad).  It is constructible by \cite{CHL}.

The parameter $X_G$ is to be considered the primary parameter.  The
transfer factor depends in a subtle way on its conjugacy class within
its stable conjugacy class. The parameters $(\bar X_H,\bar X_G)$
should be viewed as secondary, only affecting the normalization of the
transfer factor by a scalar independent of $X_G$.

The constructible function $\Delta_0$ is a step function on 
\begin{equation}\label{eqn:xfer-domain}
D:= \fh^{G-\reg}\times_Z\fg^\reg\times_Z \fh^{G-\reg}\times_Z\fg^\reg
\end{equation}
with variables $(X_H,X_G,\bar X_H,\bar X_G)$,
and it will
be defined as a finite combination of indicator functions $1_{D_{j,*}}$ of
definable sets $D_{j,{\rom{1}}}$ and $D_{j,\rom{3}}$:
\[
\Delta_0 = 
\left(\sum_{j_1=0}^{k-1} e^{2\pi i j_1/k} 1_{D_{j_1,{\rom{3}}}}\right)
\left(\sum_{j_2=0}^{k-1} e^{2\pi i j_2/k} 1_{D^{12}_{j_2,{\rom{1}}}}\right)
\left(\sum_{j_3=0}^{k-1} e^{-2\pi i j_3/k} 1_{D^{34}_{j_3,{\rom{1}}}}\right).
\]
The subscripts correspond to the Roman numerals $\rom1$ and $\rom3_1$
in the Langlands-Shelstad article.  Typically, constructible functions
have integer coefficients, but there is no harm in extending scalars
to $\ring{C}$ to allow the given roots of unity.

We will give further definable sets $D_{j,{\rom{1}}}$ and
$D_{j,\rom{3}}$.  The definable set $D_{j,\rom{3}}$ is a subset of the
$D$.  The definable set
$D_{j,\rom{1}}$ is a subset of $\fh^{G-\reg}\times_Z\fg^\reg$, which
we pull back to sets $D^{12}_{j,\rom1}$ and $D^{34}_{j,\rom1}$ on
$D$ under the projections maps onto the first two
factors, and the last two factors of $D$,
respectively:
\[
(X_H,X_G,\bar X_H,\bar X_G)\mapsto^{12} (X_H,X_G),\quad
(X_H,X_G,\bar X_H,\bar X_G)\mapsto^{34} (\bar X_H,\bar X_G).
\]

We have a space of regular nilpotent elements in the quasi-split Lie
algebra $\fg^*$.  The Kostant section \cite{Kott} to the Chevalley
quotient is determined by a choice of regular nilpotent element.  The
transfer factor will be independent of the choice of regular nilpotent
used to determine the Kostant section, allowing us to treat the
regular nilpotent element as a bound parameter.  The element $X_H\in
\fh$ determines an element $X\in \fg^*$ by the composite of maps
\[
\fh \to \fc_{\fh} \to \fc_{\fg^*} \to \fg^*,
\]
where the last map is the Kostant section.

In the rest of this section, we review the technical details of the
construction from \cite{LSxf}.  These details require a significant
amount of local notation ($\lambda(T_{\text{sc}})$, $s_\alpha$, $n(s_\alpha)$,
$m(\sigma_T)$, $a_\alpha$, $\omega_T$, $p$, $\op{inv}$, $h$, $\bar h$,
$U$, $u(\sigma)$, $v(\sigma)$, $\bar v(\sigma)$, $\kappa_T$,
$\kappa_U$) that has been borrowed from \cite{LSxf}.  This notation
will not be used elsewhere in this article.

We define $a$-data to be the collection of constants $a_\alpha:=
\alpha(X)$, for $\alpha\in \Phi$.  This choice of $a$-data allows us
to do without $\chi$-data (in the sense of Langlands and Shelstad) for
the Lie algebra transfer factor.

\subsection{The cocycle $\lambda(T_{\text{sc}})$}

We recall the definition of a Galois cocycle $\lambda(T_{\text{sc}})$ with
values in $T_{\text{sc}}(K)$, attached to the centralizer $T_{\text{sc}}$ of $X$ in
$G^*_{\text{sc}}$, where $G^*_{\text{sc}}$ is the simply connected cover of the
derived group of $G^*$.  The entire construction takes place in the
simply connected cover.  In this one subsection, we
sometimes drop the subscript ${\text{sc}}$, writing $(G^*,B^*,T^*)$ instead
of $(G^*_{\text{sc}},B^*_{\text{sc}},T^*_{\text{sc}})$, and so forth.

We take $K/\VF$ to be a splitting field of $G$, and take $L/\VF$ to be
a Galois extension that splits $T_{\text{sc}}$.  There are bound parameters
running over $L/\VF$ and its enumerated Galois group.  We let $h$ be a
bound parameter in $G^*$ such that $(T_{\text{sc}},B_{\text{sc}}) =
\op{Ad}(h)\,(T^*,B^*)$.

If $w$ is any element of the Weyl group of $T^*$ in $G^*$ with reduced
expression $w = s_{\alpha_1}\cdots s_{\alpha_r}$, we set $n(w) =
n(s_{\alpha_1})\cdots n(s_{\alpha_r})$ where
$n(s_\alpha)$ is the image
in $\op{Norm}_{G_{\text{sc}}^*}(T_{\text{sc}}^*)$ of
\[
\begin{pmatrix}0 &1\\ -1 & 0\end{pmatrix} = \op{exp} X_\alpha \op{exp} -
  X_{-\alpha} \op{exp} X_\alpha
\]
under the homomorphism $SL(2)\to G^*_{\text{sc}}$ attached to the Lie triple
$\{ X_\alpha,X_{-\alpha},H_\alpha \}$ coming from a pinning
$(B^*,T^*,\{X_\alpha\})$).  The transfer factor does not depend on the
pinning, and we treat the pinning as a bound parameter.  The
exponential map used to define $n(s_\alpha)$ is a polynomial function
on nilpotent elements.

The cocycle $\lambda(T_{\text{sc}})$ is defined using the $a$-data by an enumerated
cocycle of $\op{Gal}(K/\VF)$ as described above, with
\[
\sigma \mapsto h m(\sigma_T) \sigma(h)^{-1},\quad \text{where}\quad
m(\sigma_T):= \left( \prod_{1,\sigma}^p 
a_\alpha^{\alpha^\vee}\right) n(\omega_T(\sigma)),
\]
and where $x^{\alpha^\vee}$ denotes the image of $x$ under the coroot
$\alpha^\vee$.  An inspection of each element of the formula shows
that this cocycle is definable.  See \cite{LSxf} for the combinatorial
description of the gauge $p$ and the choice $\omega_T$ of Weyl group
elements (coming from the twisted action of the Galois group on
$T^*_{\text{sc}}$).

\subsection{Another cocycle}

Next we recall the definition of an enumerated Galois cocycle
$\op{inv}(X_H,X_G,\bar X_H,\bar X_G)$.  Let $T$ and $\bar T$ be the
centralizers in $G$ of $X$ and $\bar X$, where $X$ and $\bar X$ are
constructed as above from a Kostant section.  We take a bound
parameter space for a Galois field extension $L/K/\VF$ that splits $T$
and $\bar T$.  The cocycle takes values in the $L$-points of $U:=
(T\times \bar T)/\op{Zent}_{\text{sc}}$, where the center $\op{Zent}_{\text{sc}}$ of
the simply connected cover is mapped diagonally.

To construct the cocycle, we take bound parameters $h$ and $\bar h$
 in $G(L)$ such that
\[
\op{Ad} h\,\psi(X_G)  = X,\quad \op{Ad} \bar h\, \psi(\bar X_G) = \bar X,
\]
with $\psi$ the inner twist given above.  As an inner twist, we get an
enumerated cocycle of $\op{Gal}(K/\VF)$ by $\op{Int} u(\sigma) := \psi
\sigma(\psi)^{-1}$.  By our embedding $K\to L$ we obtain an enumerated
cocycle of $\op{Gal}(L/\VF)$.  Set $v(\sigma) = h u(\sigma)
\sigma(h)^{-1}$ and $\bar v(\sigma) = \bar h u(\sigma) \sigma(\bar
h)^{-1}$. The cocycle $\op{inv}(X_H,X_G,\bar X_H,\bar X_G)$ is then
\[
\sigma \mapsto (v(\sigma)^{-1},\bar v(\sigma))\in U(L).
\]
All this data is definable.

\subsection{Tate-Nakayama}

We continue with our fixed choice of a sufficiently divisible integer
$k$.  We have an element $\kappa\in \hat T$ that is part of our
endoscopic data.  The article \cite{CHL} shows how to treat this
element as an enumerated cocycle $\kappa_T$ with values in $X^*(T)$
for various Cartan subgroups $T$.  It uses this to give the
Tate-Nakayama pairing in terms of the Denef-Pas language.  (This
requires the cocycle space of $Z$ to include the space of pairs
$(E,\op{qFr})$ of an unramified extension of sufficiently large fixed
degree and generator $\op{qFr}$ of the Galois group $E/\VF$.)  In
particular, with the fixed choice of a highly divisible integer $k$ as
above, for $j\in \ring{Z}$, and for a definable parameter space of
enumerated cocycles $c(\sigma)$ with values in $T(L)$, there is a
definable set
\[
\op{TN}( c(\sigma),\kappa_T)_j
\]
of all cocycles such that the Tate-Nakayama pairing of the cocycle with
$\kappa_T$ has value $e^{2\pi j/k}$.  We define $D_{j,\rom1}(X_H,X_G)$ to
be
\[
\op{TN}( \lambda(T_{\text{sc}}),\kappa_{T_{\text{sc}}})_j
\]

Similarly, we may form an element $\kappa_U$ (the image of
$(\kappa_T,\kappa_{\bar T})$ in $X^*(U)$).  We let
\[
D_{j,\rom{3}}(X_H,X_G,\bar X_H,\bar X_G)
\]
 be the definable set
\[
\op{TN}( \op{inv}(X_H,X_G,\bar
X_H,\bar X_G),\kappa_{U})_j
\]
Combined with our earlier definitions, this completes the definition
of the Lie algebra Langlands-Shelstad transfer factor as a
constructible function.

Each part of the definition is a direct translation of the definition
over nonarchimedean fields, adapted to the Lie algebra.  The choice
$a_\alpha = \alpha(X)$ causes the term $\Delta_{\rom2}$ in the
Langlands-Shelstad definition to equal $1$.  For the Lie algebra
transfer factor, we may take $\Delta_{2}=1$.

\section{Transfer principles}\label{sec:tp}

\subsection{Transfer of asymptotic identities}

Our aim is to transfer identities of Shalika germs from
one field to another.  Since germs express the asymptotics of
orbital integrals, we develop a transfer principle for
asymptotic identities.  We have the following transfer principle.

\begin{theorem} \label{thm:asym}
Let $S$ be a definable set.    
Let $g$ be a constructible function on $S$ and let
  $f:S\to \ring{Z}$ be a definable function.  For each $a\in\ring{Z}$,
let $S_a = f^{-1}(a)$.  Then there exists 
  $m\in\ring{N}$ such that for all $F_1,F_2\in\Loc_m$ with the same
  residue field, we have the following: for all $a\in\ring{Z}$, 
  $g_{F_1}$ is zero on $S_a(F_1)$ if and only if $g_{F_2}$ is zero
  on $S_a(F_2)$.
\end{theorem}

\begin{proof} We use the notation of the cell decomposition theorem
  \cite[Theorem 7.2.1]{CL}.  In particular, in this one proof, $Z$
  will denote a cell, in a departure from its usual meaning as a
  cocycle space.  By quantifier elimination, we may assume that none
  of our formulas contain bound variables of the valued field sort.
  Assume that $S\subset S'[1,0,0]$, for some definable set $S'$.  By
  cell decomposition, we may partition $S$ into finitely many cells
  $Z$ such that each cell comes with a definable isomorphism
  $\iota:Z \to Z'\subset S'[1,n',n'']$ and a projection $\pi$ of the
  cell $Z'$ onto a base $C\subset S'[0,n',n'']$.  Furthermore, there
  exists a constructible function $g_C$ on each base $C$ such that
\[
g_{|Z} = p^* g_C, \quad\text{where}\quad p = \pi\circ\iota.
\]
The morphism $p$ is an isomorphism followed by a
projection of a cell to its base.  As such, there exists $m$ such that
for all $F\in \op{Loc}_m$, the map $p_F$ is onto.

Similarly, we may pick a cell decomposition of $S$ adapted to the
definable function $f:S\to \ring{Z}$.  (We use a slightly stronger
property than what is stated in \cite{CL}.  Since $f$ is definable,
$f_C$ is also definable, rather than merely constructible.)  By
\cite[Prop 7.3.2]{CL}, there is a common cell refinement that is
adapted to both $g$ and $f$.

Note that the base $C$ has one fewer valued field variable than $S$.
We iterate this construction, to eliminate all valued field variables.
After iteration, and choosing new constants $m$, $n'$, and $n''$, we
have the following situation.  $S$ can be partitioned into finitely
many definable sets $Z$, each equipped with a morphism $p:Z\to
C\subset h[0,n',n'']$.  Furthermore, for each $C$ there is a definable
function $f_C:C\to\ring{Z}$ and constructible function $g_C$ on $C$
such that
\[
f_{|Z} = p^* f_C,\quad  g_{|Z} = p^* g_C.
\]
The morphisms $p_F:Z(F)\to C(F)$ are onto, for $F\in\op{Loc}_m$.

Let $F_1,F_2\in\op{Loc}_m$ have the same residue field.  Since $g_C$ and
$f_C$ have no valued field variables, we may assume that $f_{C,F_1} =
f_{C,F_2}$ and $g_{C,F_1} = g_{C,F_2}$.

The set $S_a$ is definable.  Assume that $g_{|S_a,F_1} = 0$.  Then on
any part $Z$,
\[
0 = g_{|S_a\cap Z,F_1} = p^*_{F_1} g_{C_a,F_1}, 
\quad\text{where}\quad C_a = f_C^{-1}(a)
\]
Since $p_{F_1}$ is onto, $g_{C_a,F_1}$ is identically zero, and thus
so is $g_{C_a,F_2}$ and $0 = p^*_{F_2} g_{C_a,F_2} = g_{|S_a\cap
  Z,F_2}$.  Thus, $g_{|S_a,F_2}=0$.  This proves the theorem.
\end{proof}

\begin{cor}\label{cor:12} 
  Let $g$ be a constructible function on $S\times\ring{Z}$.  There
  exists $m$ such for all $F_1,F_2\in\op{Loc}_m$ with the same residue
  field, we have the following.  If for some $a_0\in\ring{Z}$, the
  function $g(\cdot,a)_{F_1}$ is identically zero on $S(F_1)$ for all
  $a\ge a_0$, then the function $g(\cdot,a)_{F_2}$ is also identically
  zero on $S(F_2)$ for all $a\ge a_0$.
\end{cor}

\begin{proof} Let $f:S\times \ring{Z}\to\ring{Z}$ be the projection
  onto the second factor.  The preimage of $a$ under $f$ is
  $S\times\{a\}$.  Apply the theorem to $g$ and $f$ for each $a\ge
  a_0$.
\end{proof}

We may apply the theorem and corollary to obtain a transfer principle
for asymptotic relations as follows.  Let $g$ be a constructible
function on $S$ and let $1_a$ be a definable family of functions on
$S$ indexed by $a\in\ring{Z}$.  For example, $1_a$ might be a family
of characteristic functions of a shrinking family of neighborhoods of
a point $s_0$.  (In this article, we often use the notation $1_X$
quite loosely to denote families of indicator functions {\it indexed}
by some $X$ running over a definable set. We do not require $1_X$ to
be the indicator function of $X$ itself.)  Then we obtain a
constructible function on $S\times\ring{Z}$ by $(s,a)\mapsto
g(s)1_a(s)$.  The corollary applied to this constructible function
gives a transfer principle for the vanishing of $g$ in sufficiently
small neighborhoods of the point $s_0$.  The size $a_F$ of the
neighborhood is allowed to vary with the field $F$.

\subsection{Uniformity of asymptotic relations}

We use the following lemma based on \cite[Th 4.4.4]{CGH}.

\begin{lem}  
  Let $\Lambda\times S$ be a definable set.  Let $g$ be a
  constructible function on $\Lambda\times S$.  Then there exists
  $m\in\ring{Z}$ and a constructible function $\locus{g}$ on $\Lambda$,
  such that for all $F\in\Loc_m$, the zero locus of $\locus{g}_{\,F}$
  equals the locus of
\[
\{v\in \Lambda(F)\mid \forall s\in S(F),\quad g_F(a,s)=0\}.
\]
\end{lem}

\begin{proof}  Theorem 4.4.4 of \cite{CGH} works with exponential
  constructible functions instead of constructible functions, but the
  identical proof applies when working with non-exponential functions.

  Also, that theorem is restricted to $S=h[n,0,0]$.  We show that this
  special case implies the more general statement of our lemma.
  First, the case $h[n,n',n'']$ reduces to $h[n+n'+n'',0,0]$ by
  replacing each $\ring{Z}$-variable with $\ord\, x$ for some new
  $\VF$ parameter $x$, and replacing each residue field variable with
  some $\ac\, x$ for some new $\VF$ parameter $x$.

  If $S\subset h[n,n',n'']$ is arbitrary, then we replace the function
  $g$ with the function $g 1_S$, where $1_S$ is the indicator function
  on $S$ to go from $h[n,n',n'']$ to $S$.  We are now in the context
  covered by Theorem~4.4.4.
\end{proof}

We can show that the bound $a_0$ in the Corollary~\ref{cor:12}
can be chosen uniformly in the following sense.  

\begin{thm}\label{thm:uniform}
  Let $g$ be a constructible function on a definable set
  $S\times\ring{Z}$.  Suppose that there exists $m\in\ring{Z}$ such
  that for all $F\in \op{Loc}_m$, there exists $a_F$ for which
  $g(\cdot,a)_F$ is identically zero on $S(F)$, for all $a\ge a_F$.
  Then there exists $m'\ge m$ and a single $a_0$ such that we can take
  $a_F=a_0$ for all $F\in \op{Loc}_{m'}$.
\end{thm}

\begin{proof}
  We apply the lemma with $\Lambda=\ring{Z}$ to obtain a constructible
  function $\locus{g}$ on $\ring{Z}$ that describes the locus of
  identical vanishing of $g$, which is a subset of $\ring{Z}$
  depending on $F\in\Loc_m$.

  A constructible function $\locus{g}$ lies in the tensor product
  $Q(\ring{Z})\otimes_{P^0(\ring{Z})} P(\ring{Z})$.  The exact
  definitions of $Q$ and $P^0$ are not important to us.  See
  \cite{CGH1}.  Nonetheless, it is important that this tensor product
  produces a separation of variables for $\locus{g} = \sum q_i
  \otimes p_i$, where all residue field variables occur in the ring
  $Q(\ring{Z})$ on the left, and where the ring $P(\ring{Z})$ on the
  right is the ring of constructible Presburger functions on
  $\ring{Z}$.  We may assume by cell decomposition that
  $\locus{g}$ contains no variables of valued field sort.

By the definition of constructible Presburger functions $p_i$ on
$\ring{Z}$, we may partition $\ring{Z}$ into a finite (field
independent) disjoint union of Presburger sets such that on each of
these sets, $\locus{g}$ has the form
\[
\locus{g}_F(t) = \sum_{i=1}^n c_i t^{k_i} q_F^{\ell_i t},
\]
where $q_F$ is the cardinality of the residue field of $F$,  each
$c_i$ depends only on residue field variables, the integers $k_i$ and
$\ell_i$ do not depend on any variables in the Denef-Pas language, and
we can assume that $(k_i,\ell_i)$ are mutually different for different
$i$.

The key point is that such a function can have only finitely many
zeroes, and their number is bounded by a constant that depends only on
the number $n$ of terms in the sum.  This is \cite[Lemma
2.1.7]{CGH1} and follows from $o$-minimality of the structure of the
real field enriched with the exponential.  Thus, the only way for the
zero locus of $\locus{g}$ to contain all integers greater than $a_F$
is for the coefficients $c_i$ to be zero for all fields $F$ for every
unbounded Presburger set in the disjoint union.

If we now take $a_0$ to be larger than the maximum of all the bounded
Presburger sets in the the disjoint union, then this integer works
uniformly for all fields $F$.
\end{proof}

\subsection{Transfer of linear dependence}

We use the following result from \cite{CGH2}.  We conjectured
this result to hold based on the 
requirements of smooth endoscopic matching.

\begin{theorem}[Cluckers-Gordon-Halupczok]\label{thm:cgh}
  Let $g:P\to \Lambda$ be a definable morphism between definable sets.
  For any constructible function $f$ on $P$, write $f_{\lambda}$ for
  the constructible function on the fiber $P_\lambda$ over $\lambda\in
  \Lambda$.  Let $\bf$ be a tuple of constructible functions on
  $P$. For $\lambda\in \Lambda$, let $\bf_\lambda$ be the
  corresponding tuple of constructible functions on the fiber
  $P_\lambda$.  Then there exists a natural number $m$ with the
  following property.  For every $F_1,F_2 \in \Loc_{m}$ such that
  $F_1$ and and $F_2$ have isomorphic residue fields, the following
  holds: if for each $\lambda\in \Lambda({F_1})$, the tuple of
  functions $\bf_{\lambda,F_1}$ is linearly dependent, then also
  for each $\lambda_2\in \Lambda({F_2})$, the tuple of functions
  $\bf_{\lambda_2,F_2}$ is linearly dependent.
\end{theorem}

Let $S' \subset S\times\ring{Z}$ be any sets.  (This particular
definition pertains to sets of set theory rather than definable sets.)
For $a_0\in\ring{Z}$, we say that a tuple of complex-valued functions
$(f_1,\ldots,f_r)$ is {\it $a_0$-asymptotically linearly dependent} if
for all $a\ge a_0$, the tuple of functions
$(f_{1\,|S_a},\ldots,f_{r\,|S_a})$ is linearly dependent, where
$f_{|S_a}$ denotes the restriction of $f$ to $S'\cap (S\times\{a\})$.

We need an asymptotic variant of the previous theorem.

\begin{theorem}\label{thm:cgh-asymp}
  Let $g:P\times\ring{Z}\to \Lambda$ be a definable morphism between
  definable sets, for some $P$ and $\Lambda$.  Assume that $P'$ is a
  definable subset of $P\times\ring{Z}$.  Let $\bf$ be a tuple of
  constructible functions on $P'$.  Then there exists a natural number
  $m$ with the following property.  For every $a_0\in\ring{Z}$ and
  every $F_1,F_2 \in \Loc_{m}$ such that $F_1$ and and $F_2$ have
  isomorphic residue fields, the following holds: if for each
  $\lambda\in \Lambda({F_1})$, the tuple of functions
  $\bf_{\lambda,F_1}$ is $a_0$-asymptotically linearly dependent, 
  then also for each $\lambda_2\in \Lambda({F_2})$, the tuple of
  functions $\bf_{\lambda_2,F_2}$ is $a_0$-asymptotically linearly
  dependent.
\end{theorem}

\begin{proof} We adapt the proof in \cite{CGH2}. Suppose that $\bf$ is
  an $r$-tuple.  We may replace the components $f_i$ of $\bf$ with
  $f_i 1_{P'}$ on $P\times\ring{Z}$ to reduce to the case
  $P'=P\times\ring{Z}$.  We may reduce asymptotic linear dependence to
  the identical vanishing of a determinant.  Let
\[
S = \{(\lambda,a,X_1,\ldots,X_r)\mid (X_1,a),\ldots,(X_r,a)\in P'\times_\Lambda
\cdots\times_\Lambda 
P',~g(X_i,a)= \lambda\} \subset \Lambda\times\ring{Z}\times P^{r}.
\]
We have a constructible function $d$ on $\Lambda\times\ring{Z}\times P^{r}$
given by the indicator function of $S$ times
\[
(\lambda,a,X_1,\ldots,X_r)\mapsto 
\det f_i((X_j,a)).
\]
By \cite{CGH1}, there is a constructible function $\locus{d}$ on
$\Lambda\times\ring{Z}$ whose zero locus coincides with the locus of
identical vanishing of $d$ on $(\Lambda\times\ring{Z})\times P^r$.
Let $m\in\ring{Z}$ be an integer that is sufficiently large for
Corollary~\ref{cor:12} applied to $\locus{d}$ and large enough for the
zero locus result from \cite{CGH1} to hold.

Let $F_1,F_2\in\Loc_m$ have isomorphic residue fields.  Choose any
$a_0\in\ring{Z}$.  Assume that for all $\lambda\in \Lambda(F_1)$, the
tuple of functions $\bf_{\lambda,F_1}$ is $a_0$-asymptotically
linearly dependent.  Then the determinant $d_{F_1}$ vanishes
identically for each element of $A(F_1)$, where $A:=\{(\lambda,a)\mid
a\ge a_0\}$.  The set $A$ is definable.  Then $\locus{d}_{F_1}$ is
zero on $A(F_1)$.  By Corollary~\ref{cor:12}, $\locus{d}_{F_2}$ is
zero on $A(F_2)$.  Following the preceding steps backwards now for
$F_2$ instead of $F_1$, from $\locus{d}_{F_2}$ back to asymptotic
linear dependence, we find that for all $\lambda\in\Lambda(F_2)$, the
tuple $\bf_{\lambda,F_2}$ is $a_0$-asymptotically linearly dependent.
\end{proof}

\section{Endoscopic matching of smooth functions}

\subsection{Constructible Shalika germs}\label{sec:msg}

We give a proof of Theorem~\ref{thm:lie-shalika} asserting the
existence of constructible Shalika germs on the Lie algebra.  The
Shalika germs are indexed by nilpotent orbits. The number $k$ of
nilpotent orbits can vary according to the nonarchimedean field $F$
and the cocycle $z\in Z(F)$.  Since we cannot fix $k$ in advance, this
complicates matters, and we give a $n$-tuple of constructible
``Shalika germs'' for each $n$.  When, for a given nonarchimedean
field, $n$ is larger than the actual number of nilpotent orbits, a set
$\Y$ of auxiliary parameters is empty, and the construction yields
nothing.  When $n$ is smaller than the actual number of nilpotent
orbits, we obtain an incomplete collection of germs.

This section uses the following notation.  We make fixed choices
$G^{**}$, $\Sigma$, $\cF$, and so forth, as above.
We work with respect to a fixed cocycle space $Z$ that is used to
define a form $G$ of $G^{**}$ and inner form of $G^*$.

For $n\in \ring{N}$, let $\NF^n$ be the definable subset of $n$-tuples
of pairs $\Y_i=(N_i,\s_i)$, where $\Y_i$ is a Barbasch-Moy pair, $N_i$
is a nilpotent element in $\fg_{\s_i,0}$, and such that
$N_1,\ldots,N_n$ are pairwise non-conjugate.  

For any $n$-tuple $\s=(\s_1,\ldots,\s_n)$, we write $\NF^n_\s$ for the
subset of $\NF^n$ consisting of those elements whose $n$-tuple of
second coordinates is $\s$.  For any Barbasch-Moy pair $\Y=(N,\s)$,
let $1_\Y$ be the characteristic function of $N+\fg_{\s,0+}$.

\begin{proof}[Proof of Theorem~\ref{thm:lie-shalika}]

We form the constructible orbital integrals
\[
O(X,\Y) = \int_{O(X)} 1_{\Y} \,d\mu,
\]
for $X\in\fg$ that is nilpotent or regular semisimple,
$d\mu=d\mu^{\op{nil}}$ or $d\mu^\reg$ as appropriate, and Barbasch-Moy
pair $\Y$.  If $\Y\in \NF^n$, we let $O(X,\Y)$ be the $n$-tuple whose
$i$th coordinate is $O(X,\Y_i)$.

If $\Y=((N_1,\s_1),\ldots)\in \NF^n$, we write $\Theta(\Y)$ for the
square matrix with entries
\[
O(N_i,\Y_j),\quad\text{for}\quad i,j=1,\ldots,n.
\]
Let $\Theta^a(\Y)$ be the adjugate matrix of $\Theta(\Y)$, so that
\[
\Theta^a(\Y) \Theta(\Y) = \Theta(\Y) \Theta^a(\Y) = d_n(\Y)I_n,\quad 
\text{where}\quad d_n(\Y)=\det (\Theta(\Y)). 
\]
For $X\in \fg^\reg$ and $\Y\in \NF^n$, we define the Shalika germs
$\Gamma(X,\Y)$ as the $n$-tuple given by the matrix product
\[
\Gamma (X,\Y) = \Theta^a (\Y) O(X,\Y).
\]
It then follows directly from the definition of adjugate that orbital
integrals have a constructible Shalika germ expansion
\[
d_n(\Y) O(X,\Y) = \Theta(\Y)\Gamma(X,\Y),
\]
That is, up to a determinant $d_n$, the orbital
integral of $X$ of a collection of test functions indexed by $\Y$ can
be expanded in terms of the Shalika germs weighted by the nilpotent
orbital integrals of the test functions.

If we specialize the data to a nonarchimedean field $F$, take $z\in
Z(F)$, and choose $n=k$ be the number $k$ of nilpotent orbits of $G_z(F)$,
then it follows from Theorem~\ref{thm:bm} that $\NF^n(F)$ is nonempty,
and for all $\Y\in \NF^n(F)$, we have $\det\Theta_F(\Y)$ is
nonzero. In fact, by suitable ordering of indices, the matrix is upper
triangular with nonzero diagonal entries.  We obtain the Shalika germ
expansion in its usual form.
\end{proof}

\subsection{Strategy}\label{sec:strategy}

We now prepare for the proof of Theorem~\ref{thm:local}.  This is the
most intricate proof in the article.  It will occupy
Subsections~\ref{sec:strategy} through \ref{sec:emsg}.

Let $X_H,\bar X_H,\bar X_G$ be parameters that appear in the
transfer factor.  For a given $X_H$, we let $X_G$ be the image of $X_H$ in
$\fc_{\fg^*}$ under the morphism
\[
\fh \to \fc_{\fh} \to \fc_{\fg^*},
\]
and let $O^{st}(X_G)$ be the preimage of $X_G$ in $\fg$.  In what
follows, the stable orbit $O^{st}(X_G)$ is assumed always to be
derived from the parameter $X_H$ in this manner.  As described above,
we have invariant motivic measures $d\mu^\reg_G$ and $d\mu^\reg_H$ on
$O^{st}(X_G)$ and $O^{st}(X_H)$.

We define $\kappa$-orbital integrals by the following equation:
\begin{equation}\label{eqn:kappa}
O^\kappa(X_H,\bar X_H,\bar X_G,f) = \int_{x\in O^{st}(X_G)}
\Delta(X_H,x,\bar X_H,\bar X_G) f (x)\,d\mu^\reg_G,
\end{equation}
where $f$ is a constructible function on $\fg$ and the parameters
$(X_H,\bar X_H,\bar X_G)$ run over the definable set $V$ in Equation
\ref{eqn:delta-domain} (omitting the factor for $X_G$).  Similarly, we
define stable orbital integrals on $\fh$ by
\begin{equation}\label{eqn:stable}
SO(X_H,f^H) = \int_{x\in O(X_H)} f^H (x)\,d\mu^\reg_H,
\end{equation}
where $f^H$ is a constructible function on $\fh$.

Waldspurger has proved that the germs of $O^\kappa$ are linear
combinations of the germs of $SO$ in Equation~\ref{eqn:stable} when
the field $F$ has characteristic zero~\cite{waldspurger1997lemme}.  We
wish to transfer these linear relations among germs to positive
characteristic.

The basic naive strategy is to invoke Theorem~\ref{thm:cgh-asymp}
directly to transfer these linear relations among germs.
Unfortunately, the naive strategy does not work, because the
coefficients of the linear relation potentially vanish identically on
part of the cocycle space $Z$, which would mean that we would obtain
no information about the endoscopic matching for some of the twisted
forms of a reductive group.  We note that it is not always possible to
isolate definably a single isomorphism class of reductive groups $G$.
See Section~\ref{sec:classification}.

The next (less naive) strategy is to take a product $G =
G_1\times\cdots\times G_r$ over all of the groups up to isomorphism
that appear as fibers over the cocycle space $Z$.  On this product,
the corresponding cocycle space does in fact determine a single group
up to isomorphism.  The germs on this product are the products of
germs of the factors.  We can invoke the theorem to transfer linear
relations for this product of groups.  Unfortunately, this strategy
also fails, because all of the $\kappa$-Shalika germs might vanish
identically for one of the factors; and when this happens, we cannot
derive any information from the product about the other individual
factors $G_i$.

This nonetheless, leads to a strategy that works.  Again we take
$r$-fold products of factors, but we enhance our collection of factors
to include the stable Shalika germs on the quasi-split endoscopic
groups (Equation~\ref{eqn:psi}).  We know how to choose stable Shalika
germs that are nonzero.  Indeed the stable Shalika germ of the regular
nilpotent class in a quasi-split inner form is nonzero.  Thus, we are
able to avoid the bad situation where one of the factors vanishes
identically.  This works.

The proof will use the following elementary facts about tensor
products of finite dimensional vector spaces in our analysis of the
products of Shalika germs.  The second statement may be used to
extract a nontrivial linear relation on a single factor from a linear
relation on the tensor product.

\begin{lem}\label{thm:tensor}
Let $V_1,\ldots,V_r$ be finite dimensional vector spaces of
dimensions $\dim(V_i) = n_i$.  
\begin{enumerate}
\item 
Let $S_i\subset V_i$ be a finite set of vectors for each
  $i = 1,\ldots,r$. 
If the product of the cardinalities of the
  sets $S_i$ is $n_1\cdots n_r$ and if the tensors $w_{i_1}\otimes
 \cdots\otimes w_{i_r}$, for $w_{i_j}\in S_j$, span 
$V_1\otimes\cdots\otimes V_r$, then
  for each $i$, the set $S_i$ is a basis of $V_i$.
\item For each $i=1,\ldots,r$, let $W_i$ be a subspace of $V_i$.  Let
  $v_{1}\otimes \cdots \otimes v_{r} \in W_1\otimes\cdots\otimes W_r$,
  with $v_{i}\in V_i$.  If for some $i$, we have $v_{j}\ne0$, for all
  $j\ne i$, then $v_{i}\in W_i$.
\end{enumerate}
\end{lem}

\begin{proof} This is elementary multilinear algebra.
\end{proof}

In what follows, to use this lemma, we regularly identify finite
dimensional subspaces of $C(U_1\times\cdots\times U_r)$ of the
continuous functions on $U_1\times\cdots U_r$, for various choices of spaces
$U_i$, with corresponding finite dimensional subspaces of the tensor
product $C(U_1)\otimes \cdots\otimes C(U_r)$, under the map
\[
C(U_1)\otimes \cdots\otimes C(U_r)\to C(U_1\times\cdots\times U_r),
\]
often without explicit mention.  We recall that the orbital integral
of $f_1\otimes\cdots\otimes f_r$ on a product
$\fg_1\times\cdots\times \fg_r$ is a product of the corresponding
orbital integrals on the factors.

\subsection{Parameter sets}

We establish notation and various parameters that will be used in
the proof of Theorem~\ref{thm:local}.  We make all the fixed choices
that are established in Section~\ref{sec:fixed}.  Let $(G,H)_{/Z}$ be
a definable reductive group $G$ with definable endoscopic group $H$
over a cocycle space $Z$.  We remain in this context until the end of
Section~\ref{sec:emsg}.

\subsubsection{Finiteness of the parameter set}\label{sec:finiteness}

The order in which we construct our parameters is extremely important
in the proof.  Later, in Subsection~\ref{sec:field}, we will make
choices of constants $r\in\ring{N}$, and $r$-tuples: $k$, $k'$, $\s$,
and $\s'$ that will be field dependent.  Before we do that, in this
subsection, we bound these parameters so that they run over finitely
many possibilities.  That is, we define a finite set $\Xi$ that
contains all possible parameter choices $\xi=(r,k,k',\s,\s')\in\Xi$.
It is important to work with the full finite set $\Xi$ of
possibilities until Subsection~\ref{sec:field}.

In this paragraph, we define a finite set $\Xi$ of parameters
$\xi=(r,k,k',\s,\s')$ by placing boundedness constraints on
$r,k,k',\s,\s'$.  (We do not need sharp bounds in this paragraph; we
are simply establishing a finiteness result.)  We constrain
$r\in\ring{N}$ to be at most the lim sup over $m$ of maximum over
$F\in\Loc_m$ of $r_F$, where $r_F$ is the number of isomorphism
classes of reductive groups $G_z$ obtained as $z$ runs over $Z(F)$.
That is, it should be at most the largest number of isomorphism
classes of reductive groups in the family $G$ in large residual
characteristic.  The $r$-tuple $k\in\ring{N}^r$ has components $k_i$
that are constrained to be at most the largest number of nilpotent
orbits there can be in sufficiently large residual characteristic in a
Lie algebra of some group in the family $G$.  (See
Corollary~\ref{thm:nilbound}.) We constrain the tuple $\s$ to be an
$r$-tuple, where each component $\s_i$ is a $k_i$-tuple of elements of
$\cF_G$. Recall that $\cF_G$ is a fixed choice that does not depend on
the cocycle space $Z$.  The cardinality of $\cF_G$ is finite and field
independent.  The tuple $k'\in\ring{N}^r$ is similarly constrained to
have components $k'_i$ that are at most the largest number of
nilpotent orbits there can be in sufficiently large residual
characteristic in a Lie algebra of some group in the family of
endoscopic groups $H$.  The $r$-tuple $\s'$ is constrained as $\s$,
but using $k'$ and $\cF_H$.  By imposing finiteness constraints as we
have on $r$, $k$, $k'$, $\s$, and $\s'$, the parameters
$\xi=(r,k,k',\s,\s')$ range over a finite set $\Xi$ of possibilities.

\subsubsection{Parameter sets that depend on $\xi$}

Let $\xi=(r,k,k',\s,\s')\in\Xi$ be any parameter in
$\Xi$.  Until Subsection~\ref{sec:field}, the element $\xi$ remains
a fixed but arbitrary element of $\Xi$.  We often omit $\xi$ and its
(fixed but arbitrary) components $r$, $k$, $k'$, $\s$, and $\s'$ from
the notation.  We will give definable sets $P$, $\Lambda$, a definable
morphism $g:P\times\ring{Z}\to \Lambda$, and a tuple of constructible
functions $\bf$ on $P\times\ring{Z}$.  Let $Z^r$ be the definable set
of $r$-tuples $(z_1,\ldots,z_r)$, where $z_i \in Z$ and such that
$z_i$ and $z_j$ are not cohomologous for $i\ne j$.

The Lie algebras $\fh_z$, $\fg_z$ depend on a parameter $z\in Z$.
For $i=1,\ldots,r$, we have spaces
\begin{align*}
\Lambda^1_i = &\{(z,\bar X_H,\bar X_G,\Y_G,\Y_H) \mid \\
     &\quad z\in Z, 
\bar X_H \in \fh_z, \bar X_G\in \fg_z, 
\Y_G\in \NF^{k_i}_{G_z,\s_i}, \Y_H\in \NF^{k'_i}_{H_z,\s'_i}, \Y_{H,1}:\op{reg}\}.
\end{align*}
(We apologize for the nesting of subscripts.)  The condition
$\Y_{H,1}:\op{reg}$ means that we constrain the nilpotent part $N$ of
the Moy-Prasad pair $\Y_{H,1}$ to be a regular element.  We write
$z(\lambda_i)$, $\Y_G(\lambda_i)$, and so forth for the components of
$\lambda_i\in\Lambda^1_i$.  For $i=1,\ldots,r$, we set
\[
\Lambda^r = \Lambda^r_\xi = 
\{(\lambda_1,\ldots,\lambda_r) \mid \lambda_i \in \Lambda^1_i\quad
  (z(\lambda_1),\ldots,z(\lambda_r))\in Z^r \}.
\]
If $j\in\ring{N}$ and $\epsilon\in\{\pm\}$, write
\[
\Y(\lambda_i,(j,\epsilon)) = 
\begin{cases}
\Y_G(\lambda_i)_j,&\text{if }\epsilon= -,\quad j\le k_i,\\
\Y_H(\lambda_i)_j,&\text{if }\epsilon= +,\quad j\le k'_i.
\end{cases}
\]
We
define:
\[
P^1_i = \{(X_H,\lambda_i) \mid \lambda_i\in \Lambda^1_i,
\quad z = z(\lambda_i),\quad X_H \in \fh_z\}.
\]
Also, set:
\[
P^r = P^r_\xi = \{(X_H,\lambda) \mid \lambda = 
(\lambda_1,\ldots,\lambda_r)\in \Lambda^r_\xi,\quad
  X_H = (X_{H,1},\ldots,X_{H,r}), \quad (X_{H,i},\lambda_i) \in P^1_i\}.
\]
We have a projection map $P^r\to \Lambda^r$ onto the second coordinate.

Set
$L_i = L^+_i\sqcup L^-_i$, where $L^-_i =
\{1,\ldots,k_i\}\times\{-\}$ and $L^+_i =
\{1,\ldots,k'_i\}\times\{+\}$.
Set
\[
L = \{\ell = (\ell_1,\ldots,\ell_r)\mid \ell_i\in L_i\},\quad\text{and}\quad
L^+ = \{\ell = (\ell_1,\ldots,\ell_r)\mid \ell_i\in L^+_i\}.
\]

We introduce a parameter $a\in\ring{Z}$ to construct germs of
functions as follows.  Let $1_a$, for $a\in\ring{Z}$ be a sequence of
support functions, defined as the preimage in $\fg$ of indicator
functions of small balls around $0$ in $\fc_{\fg^*}$ tending to $0$ as
$a$ tends to infinity.  We take the preimage of these balls in
$\fc_{\fh}$ and $\fh$ and denote that support sequence $1_a$ as well,
by abuse of notation.  We may truncate the orbital integrals using
these functions: $1_a SO(X,f^H)$ and $1_a O^\kappa(X_H,\bar
X_H,\bar X_G,f)$.  Integrals now have an extra integer parameter
$a\in\ring{Z}$ that we may use to study $a$-asymptotic linear
dependence.  

We define constructible functions.  For $(X_H,\lambda)\in
P^r_\xi$,  for $i=1,\ldots,r$, and $\ell_i\in L_i$,
set
\begin{equation}\label{eqn:psi}
\psi_{i,\ell_i}(X_{H,i},\lambda_i) = 
\begin{cases}   
  O^\kappa(X_{H,i},\bar X_H(\lambda_i),\ldots,1_{\Y(\lambda_i,\ell_i)}),
   & \text{if } \ell_i\in L^-_i \\
  SO(X_{H,i},1_{\Y(\lambda_i,\ell_i)}),
   & \text{if } \ell_i\in L^+_i.
\end{cases}
\end{equation}
We also define a tuple $\bf=(\ldots,1_a f_\ell,\ldots)$ of constructible
functions on $P^r_\xi\times\ring{Z}$ indexed by $\ell\in L$:
\[
((X_H,\lambda),a)\mapsto 1_a f_\ell(X_H)_\lambda 
:= 1_a \prod_{i=1}^r \psi_{i,\ell_i}(X_{H,i},\lambda_i).
\]

In summary, the tuple $\bf$ is indexed by $\ell$ which is a tuple
that, when $\ell$ is in $L^-$, enumerates the nilpotent orbits in
$\fg_{z(\lambda_1)}\times \cdots\times \fg_{z(\lambda_r)}$, and when
$\ell$ is in $L^+$, enumerates the stable nilpotent distributions on
$\fh_{z(\lambda_1)}\times \cdots\times \fh_{z(\lambda_r)}$.  The
components of $\bf$ (the functions $f_\ell$) mix the reductive group
and its endoscopic group together.  Each function $f_\ell$ is an
orbital integral on $\fg'_{z(\lambda_1)}\times \cdots\times
\fg'_{z(\lambda_r)}$, where each factor $\fg'_{z(\lambda_i)}$ is
$\fh_{z(\lambda_i)}$ or $\fg_{z(\lambda_i)}$, depending on which side
of the disjoint union the parameter $\ell_i\in L_i = L^+_i\sqcup
L^-_i$ lands.  These orbital integrals are evaluated at the
($r$-tuples of) test functions associated with Barbasch-Moy pairs as
in Section~\ref{sec:msg}. 

The condition $\Upsilon_{H,i,1}:\op{reg}$ (in the
definition of $\Lambda^1_i$) ensures that one of the nilpotent classes
on each endoscopic factor is regular, which will be needed later to
make sure that the corresponding stable Shalika germ does not vanish.
The mixing of factors $\fg'_{z(\lambda_i)}$ allows us to take the
regular nilpotent (nonzero) stable germ from the endoscopic group on every factor
but one and an arbitrary $\kappa$-germ on the remaining factor in
order to isolate a given $\kappa$-germ.

\subsubsection{Field dependent choices}\label{sec:field}

Recall that $\xi=(r,k,k',\s,\s')$ denotes an arbitrary element in the
finite set $\Xi$.  For each $\xi\in\Xi$, we have produced definable
sets $P=P^r_\xi$, $P'=P\times\ring{Z}$, $\Lambda=\Lambda^r_\xi$, a
definable morphism $g:P\times\ring{Z}\to \Lambda$, and a tuple of
constructible functions $\bf$ on $P'$.  For each $\xi$, we invoke
Theorem~\ref{thm:cgh-asymp} on these data to obtain a constant
$m(\xi)\in\ring{N}$ with properties as described in the theorem.
Recall that $m(\xi)$ is the parameter that gives the restriction on
the residual characteristic in the nonarchimedean field.  We let $m$
be the maximum of the constants $m(\xi)$ as $\xi$ runs over the finite
set $\Xi$.

As a result of this discussion, we have a constant $m$ that is not
tied to any particular parameter $\xi$.  This constant will be used to
satisfy the claim in the statement of Theorem~\ref{thm:local} of the
existence of a constant $m$ with desirable properties.

Let $F_2\in\op{Loc}_m$.  Choose a second field $F_1\in \op{Loc}_m$
that has characteristic zero and a residue field isomorphic to that of
$F_2$.  We will use the transfer principle between these two fields.

Now that $m$, $F_1$, and $F_2$ have been constructed, we make a
particular choice of $\xi\in\Xi$.  (In particular, from this point
forward, $\xi=\xi_{m,F_1,F_2}$ and its components $r,k,k',\s,\s'$ will no longer
denote arbitrary elements drawn from a finite set $\Xi$, but will
denote choices of parameters that have been adapted to the specific
nonarchimedean context of $m$, $F_1$, and $F_2$.)

We start with $r$.  Let $r$ be the
maximum for which $Z^r(F_1)\ne\emptyset$.  That is, it is the number
of nonisomorphic reductive groups in the family of cocycles.  This
maximum can be expressed as a sentence in the Denef-Pas language and
can hence be transferred to $F_2$ by an Ax-Kochen style transfer
principle.  Thus, $r$ is also the maximum for which
$Z^r(F_2)\ne\emptyset$.


Pick any $z\in Z^r(F_2)$.  Write $G_i=G_{z_i}$.  Let
$k=(k_1,\ldots,k_r)$ be a tuple such that $k_i$ is the number of
nilpotent orbits in $\fg_{z_i}(F_2)$.  Choose a $k_i$-tuple $\s_i\in
\cF_{G_i}^{k_i}$ such that $\NF^{k_i}_{G_i,\s_i}(F_2)\ne \emptyset$.
Let $\s = (\s_1,\ldots,\s_r)$.

Finally, we choose parameters $k'$ and $\s'$ as follows.  Pick $z\in
Z^r(F_1)$.  By a transfer principle, we may assume that $z$ is
selected in such a way that $\NF^{k_i}_{G_i,\s_i}(F_1)\ne \emptyset$,
where $G_i=G_{z_i}$, for $i=1,\ldots,r$.  Write $H_i = H_{z_i}$. Let
$k'=(k'_1,\ldots,k'_r)$ be a tuple such that $k'_i$ is the dimension
of the space of stable Shalika germs on $\fh_{z_i}(F_1)$.  By the
definition of stable distribution, this equals the dimension of the
space of stable orbital integrals supported on the nilpotent set.  For
each $i$, let $\mu_{ij}^{st}$, $j = 1,\ldots,k'_i$, be a basis of the
stable distributions supported on the nilpotent set of $\fh_{z_i}(F_1)$.
We have a Shalika germ expansion
\[
SO(X_H,f^H) = \sum_{j=1}^{k'_i} \Gamma^{st}_{ij}(X) \mu_{ij}^{st}(f^H),
\]
for compactly supported functions $f^H$ on $\fh_{z_i}(F_1)$ and some
stable germs $\Gamma^{st}_{ij}$.  The germ expansion holds on $G$-regular
semisimple elements in a suitable neighborhood (depending on $f^H$).

For each $i\le r$, we define a tuple $\Y_i :=
(\Y_{i1},\ldots,\Y_{ik'_i})\in \NF^{k'_i}_{H_{i}}(F_1)$ as follows.
We pick parameters $\Y_{ij}$ such that the functions $1_{\Y_{ij}}$,
for $j=1,\ldots,k'_i$, form a dual basis to $\mu_{ij}^{st}$ (with $i$
fixed and $j$ variable).  (The indicator functions for Moy-Prasad
pairs span the space dual to invariant distributions with nilpotent
support, hence span the quotient space dual to stably invariant
distributions with nilpotent support; hence a subset of the functions
forms a dual basis.)  We can do this by Theorem~\ref{thm:bm}.  Note
that there is a unique stable distribution supported on the regular
nilpotent set.  We may pick the dual space in a way that the nilpotent
component of $\Y_{i1}$ is regular.

Let $\s'_i$ be the tuple of second coordinates of $\Y_i$.
Set $\s' = (\s'_1,\ldots,\s'_r)$.  We have now fixed the parameter
$\xi=(r,k,k',\s,\s')$.
 
\subsection{Endoscopic matching of Shalika germs}\label{sec:emsg}

We now give the proof of Theorem~\ref{thm:local}, adapted to Lie
algebras.  In Section~\ref{sec:adapt}, we show that the proof can be
easily adapted to the groups.  Here, we work in the context of the
previous subsection.

\begin{proof}[Proof of Theorem~\ref{thm:local}, adapted to Lie algebras]
We return to our earlier notation of $z$ as a
parameter running over the definable set $Z^r$ (as opposed to an $F$-point
of that set).

We write $\Lambda^r(F_1) = \Lambda_A \cup\Lambda_B$, where
\[
\Lambda_A = \{\lambda\in\Lambda^r(F_1)\mid
\exists a_0\in\ring{Z},\ (\ldots,1_af_{\ell,\lambda,F_1},\ldots), \text{ for } \ell\in L^+,
\quad\text{is } a_0\text{-asymp. lin. dep.} \}.
\]
and $\Lambda_B = \Lambda^r(F_1)\setminus \Lambda_A$.  

We claim that the number $a_0$ that appears in the definition of $\Lambda_A$
can be chosen to be independent of $\lambda\in\Lambda_A$.  Indeed,
there are only finitely many different situations up to isomorphism,
and we can simply choose the maximum $a_0$ among finitely many cases.

For $\lambda\in \Lambda_B$, there does not exist an $a_0$ for which
the given tuple of functions is $a_0$-asymptotically dependent.  Note
that the number of stable distributions with nilpotent support on the
product of $\fh_{z(\lambda_i)}(F_1)$ does not depend on $\lambda\in
\Lambda^r(F_1)$, since different choices of $\lambda$ merely permute
the factors $\fh_{z(\lambda_i)}(F_1)$.  This number is $k'_1\cdots k'_r$,
the cardinality of $L^+$.  We invoke Lemma~\ref{thm:tensor}(1) to see
that for all $i$, the functions
\begin{equation}\label{eqn:1a}
1_a\psi_{i,\ell_i}(\cdot,\lambda_i)_{F_1},\quad \text{ for }\ell_i\in L^+_i,
\end{equation}
are also not $a_0$-asymptotically dependent for any $a_0$.  Thus, the
functions $1_a 1_{Y(\lambda_i,\ell_i),F_1}$ span the space dual to the
space of stable nilpotent distributions on $\fh_{z_i(\lambda)}(F_1)$.
We use Waldspurger's fundamental result \cite{waldspurger1997lemme}
that the $\kappa$-orbital integral (over $F_1$ of characteristic zero)
admits smooth matching. This implies that for any test function on any
factor $\fg_{z_i(\lambda)}$, there exists $a_0$ such that the
$\kappa$-orbital integrals of the test function satisfy an
$a_0$-asymptotic relation with the functions in Formula \ref{eqn:1a}.
If we constrain the test functions to be $1_\Y$ for Moy-Prasad pairs,
then as we vary the parameters, there are only finitely many
situations up to isomorphism.  Thus, we may pick a single $a_0$ that
works for all $\lambda\in\Lambda_B$.

We claim that for all ${\bar\ell} \in L\setminus L^+$, there exists
$a_0\in\ring{Z}$, such that for each $\lambda\in \Lambda^r(F_1)$, the
tuple of functions $1_{a_0,F_1}f_{\ell,\lambda,F_1}$, indexed by
$\ell\in \{{\bar\ell}\}\cup L^+$, is $a_0$-asymptotically linearly
dependent.  We may check this claim separately for $\Lambda_A$ and
$\Lambda_B$ and take the maximum of the constants $a_0$ that we
obtain.  For $\Lambda_A$ we have already checked that such a constant
can be chosen, even without including $\bar\ell$.  For $\Lambda_B$, we
have checked that we have an $a_0$-asymptotic linear dependency that
expresses the function with index $\bar\ell_i$ in terms of the others
for each $i=1,\ldots,r$. By taking tensor products of the relations,
we obtain a linear relation on the tuple of functions indexed by $\ell\in
\{{\bar\ell}\}\cup L^+$.  This gives the claim.

In fact, we may take $a_0\in\ring{Z}$ to be independent of $\bar\ell$,
because ${\bar\ell}\in L\setminus L^+$ runs over finitely many
possibilities.

We apply Theorem~\ref{thm:cgh-asymp} to the morphism
$f:P^r\times\ring{Z}\to \Lambda^r$.  By the theorem, for each ${\bar\ell}\in
L\setminus L^+$, and for all $\lambda\in \Lambda^r({F_2})$, the
functions $f_{\ell,\lambda,F_2}$, for $\ell\in \{{\bar\ell}\}\cup L^+$ are
$a_0$-asymptotically linearly dependent.  Let $c_{{\bar\ell},\ell,\lambda}\in
\ring{C}$ be the coefficients of a nontrivial linear combination
(depending on $\lambda$).  We write this relation as
\begin{equation}\label{eqn:du}
\sum_{\ell\in \{{\bar\ell}\}\cup L^+} c_{{\bar\ell},\ell,\lambda}1_af_{\ell,\lambda,F_2}  =
0\quad\text{for }\quad a\ge a_0.
\end{equation}
Note that we may assume that
$\op{supp}(1_{a'})\subseteq\op{supp}(1_a)$ if $a'\ge a$, so that when
we pass to a larger $a\in\ring{Z}$, a given linear relation still
holds on the smaller support. Thus, we may take the coefficients
$c_{{\bar\ell},\ell,\lambda}$ to be independent of $a$.  For each
${\bar\ell}$ and $\lambda$, the nontriviality of the relation means
that $c_{{\bar\ell},\ell,\lambda}\ne 0$ for some $\ell$.

We claim that there exists $\lambda\in \Lambda^r({F_2})$, such that
for all ${\bar\ell}\in L\setminus L^+$,
$c_{{\bar\ell},{\bar\ell},\lambda}\ne 0$.  Otherwise, for all
$\lambda\in \Lambda^r(F_2)$, there exists ${\bar\ell}\in L\setminus
L^+$ such that $c_{{\bar\ell},{\bar\ell},\lambda}=0$.  For these
choices, Equation~\ref{eqn:du} becomes
\begin{equation}\label{eqn:du2}
\sum_{\ell\in L^+} c_{{\bar\ell},\ell,\lambda}1_af_{\ell,\lambda,F_2}  = 0.
\end{equation}
This asserts that for all $\lambda\in \Lambda^r(F_2)$, the functions
$1_af_{\ell,\lambda,F_2}$, for $\ell\in L^+$, are linearly dependent
for $a\ge a_0$.  Applying Theorem~\ref{thm:cgh-asymp} again in the
opposite direction, to go from dependence on $F_2$ to linear
dependence on $F_1$, we find that for all $\lambda'\in
\Lambda^r(F_1)$, the functions $1_af_{\ell,\lambda',F_1}$, for
$\ell\in L^+$ are linearly dependent.  There exists
$\lambda'\in\Lambda^r(F_1)$ such that $z(\lambda')\in Z^r(F_1)$ is the
parameter used to construct $k'$ and $\s'$.  This contradicts the
choice of $k',\s'$, which were chosen to give linear independence.
This gives the claim.

Let $\lambda\in \Lambda^r(F_2)$ be the parameter in the claim.  The
$r$-tuples $k$ and $\s$ were constructed using some different
parameter $z\in Z(F_2)$.  However, different choices of cocycles
merely permute the constants $k_i$.  Hence, the product $k_1\cdots
k_r$ is independent of $z$.  Since $\Y_G(\lambda_i)\in
\NF^{k_i}:=\NF^{k_i}_{G_i,\s_i}(F_2)$, where $G_i = G_{z(\lambda_i)}$,
and since the $\NF^{k_i}$ enforces non-conjugacy conditions on its
components, we see $G_i(F_2)$ has at least $k_i$ nilpotent orbits.
Since the product $k_1\cdots k_r$ is fixed, in fact, $G_i(F_2)$ has
exactly $k_i$ nilpotent orbits.  We conclude that the tuples $k$ and
$\s$ are compatible with the parameter $z(\lambda)\in Z(F_2)$.

We write $\Y_H(\lambda)=(\Y_1,\ldots,\Y_r)$.  For each $i=1,\ldots,r$,
the definition of $\lambda_i\in\Lambda^1_i$ constrains $\Y_{i1}$ to be
associated with a regular nilpotent orbit.  Thus, the index $(1,+)\in
L^+_i$ represents the regular nilpotent orbit in the Lie algebra
$\fh_{z_i}(F_2)$ of the endoscopic group.

The nonvanishing $c_{{\bar\ell},{\bar\ell},\lambda}\ne 0$, gives for
each ${\bar\ell}\in L\setminus L^+$ a linear relation for
$1_af_{{\bar\ell},\lambda,F_2}$ in terms of products of stable orbital
integrals on factors $\fh_{z_i}(F_2)$.  In particular, for each
$i\in\{1,\ldots,r\}$, and for each $j\in\{1,\ldots,k_i\}$, we define
${\bar\ell} = {\bar\ell}(i,j) = ({\bar\ell}_1,\ldots,{\bar\ell}_r)$,
where
\[
{\bar\ell}_{i'} = \begin{cases} (1,+), &\text{if } i'\ne i\\
       (j,-) & \text{if } i' = i.
  \end{cases}
\]
It is known by Langlands and Shelstad that when suitably normalized,
the Shalika germ of the stable regular nilpotent class is identically
$1$~\cite{LSxf}.  In particular, it is nonzero.  Considering the
linear relation involving $1_af_{{\bar\ell},\lambda,F_2}$ as a
function of $X_{H,i}$ alone, it gives a nontrivial linear relation
between functions
$X_{H,i}\mapsto\psi_{i,\bar\ell_i}(X_{H,i},{\lambda})$ (with nonzero
coefficient) and the stable nilpotent orbital integrals on
$\fh_{z_i}(F_2)$.  Lemma~\ref{thm:tensor}(2) extracts the linear
relation on a factor from a linear relation on the tensor product.  As
we run over all $j$ and $i$, we obtain the desired matching of all
$\kappa$-Shalika germs on all reductive groups in fibers over the
cocycle space $Z(F_2)$.  This proves the theorem.
\end{proof}

\subsection{Adaptation from nilpotence to unipotence}
\label{sec:adapt}

We now adapt the results of the previous two sections to unipotent
classes in the group.

\begin{proof}[Proof of Theorem \ref{thm:gtf}]
  All of the proofs go through with the following changes.  Functions
  on the Lie algebra (Shalika germs, orbital integrals) are replaced
  with functions on a neighborhood of the identity element in the
  group.  We use notation $\gamma_G,\gamma_H,\bar \gamma_G,\bar
  \gamma_H$, and so forth for elements in the appropriate groups,
  instead of $X_G,X_H,\bar X_G,\bar X_H$.

The choice of $a$-data becomes $a_\alpha = \alpha(\gamma)^{1/2} -
\alpha(\gamma)^{-1/2}$ (instead of $\alpha(X)$).  For this, we
restrict $\gamma$ to the set of topologically unipotent elements (in large
residual characteristic), on which square roots can be extracted.

We may assume that the multiplicative characters that occur in the
transfer factor are trivial near the identity element. Near the
identity element, we have $\Delta_{\rom{3}_2}=1$.

The Kostant section is replaced with the Steinberg section in the
quasi-split inner form.  In the construction of transfer factors in
\cite{CHL}, the compatibility of the Kostant section with the transfer
factor is used, as proved by Kottwitz \cite{Kott}.  In this article we
have not used the compatibility with transfer factors, and we do not
need to know whether the Steinberg section is compatible with the
transfer factor.  In fact, we could replace the explicit Steinberg
section with an existential assumption of a section.  This is because
the transfer factor is independent of the choices made in the
quasi-split inner form.
\end{proof}

\begin{proof}[Proof of Theorem \ref{thm:local} for groups] 
  The proof already given in Section~\ref{sec:emsg} for the Lie
  algebra is readily adapted to the group.

  Nilpotent elements are replaced with unipotent elements.  In large
  residual characteristic, the exponential map is defined on the set
  of nilpotent elements and gives a bijection between nilpotent
  classes in the Lie algebra and unipotent elements in the group.  The
  exponential map is polynomial and shows that the set of unipotent
  elements is a definable set.

  The set of parahoric subalgebras is replaced with the set of parahoric
  subgroups.  We use Theorem~\ref{thm:bmg} for the properties of the
  Barbasch-Moy pairs.  We have structured our proofs in the nilpotent
  case in such a way that the proofs in the unipotent case apply
  with minimal changes.
\end{proof}

\subsection{Endoscopic matching of smooth functions on the group}
\label{sec:xfer}

In this section, we prove Theorem~\ref{thm:xfer}. 

In \cite{LSxf}, Langlands and Shelstad define the notion of matching
(transfer) of orbital integrals from a reductive group over a
nonarchimedean field to an endoscopic group.  In particular, they
introduce the transfer factors that are used for the matching.  In
\cite{LSd}, they take the further step of reducing the existence of
matching to a local statement at the identity in the centralizer of a
semisimple element.  We review this reduction below.  It is this
reduction that we will use in this section, in combination with the
endoscopic matching of Shalika germs from the previous section, to
establish endoscopic matching in sufficiently large characteristic.

In earlier sections, the primary focus was on constructible
functions.  In this section, the focus is the $p$-adic theory.

We recall some results from \cite{LSd}.  That article assumes that $F$
is a nonarchimedean field of characteristic zero.  We need to relax the
restriction on the characteristic to allow fields of large positive
characteristic.  We briefly sketch how the assumption on the
characteristic enters into the article \cite{LSd} and how it can be
relaxed.  The assumption of characteristic zero is used in the parts
of the article devoted to the archimedean places.  These sections of the
article do not concern us here.  The assumption is used in an essential
way in \cite{LSd}*{\S 6.6 Case 3} in a global argument used to deal
with fields of even residual characteristic.  Since we are assuming
large residual characteristic, we may disregard that case. The
complete reducibility of a module is used \cite{LSd}*{\S 5.5}. Again,
it is enough to assume the characteristic is large.  The cocycle
calculations, which form the bulk of the article, do not require the
assumption that the field has characteristic zero.

That article also cites other sources that assume that the
characteristic is zero.  We need Harish-Chandra's descent of orbital
integrals near a semisimple element, which holds in positive
characteristic by \cite{adler-korman}*{\S 7}. They apply descent to
Harish-Chandra character of a representation; but as they point out,
it applies in fact to any $G$-invariant distribution. Their result
relies on Harish-Chandra's submersion principle, which in positive
characteristic is proved by Prasad in
\cite{adler-debacker:mktheory}*{Appendix B}.  The article cites
\cite{LSxf}, which also assumes the field has characteristic zero. But
here again, the arguments are readily adapted to large
characteristics.

Let $G$ be a reductive group over a nonarchimedean field $F$,
and let $H$ be an endoscopic group of $G$.  We recall \cite[\S
2.1]{LSd} that $(G,H)$ admits {\it endoscopic $\Delta$-matching} if
for each $f\in C_c^\infty(G(F))$ there exists $f^{\tilde H}\in
C_c^\infty(\tilde H(F),\theta)$ such that $f$ and $f^{\tilde H}$
have\footnote{Langlands and Shelstad fix $\tilde{\bar \gamma}_H$ and
  $\bar\gamma_G$ in the transfer factor and then drop them from
  notation, so that only the first two variables of $\Delta$ are
  displayed.  Throughout most of \cite{LSxf}, they assume that the
  split extension of $\hat H$ by $W_F$ (the Weil group of $F$) is an
  $L$-group (\S1.2).  When they drop the assumption in \S4.4, the
  general transfer factor is denoted $\Delta$ and the local factor
  becomes denoted $\Delta_{loc}$.  We follow their conventions here.}
$\Delta$-matching orbital integrals:
\[
SO(\tilde \gamma_H,f^{\tilde H}) 
 = \sum_{\gamma_G} \Delta(\tilde\gamma_H,\gamma_G)O(\gamma_G,f)
\]
for all strongly $G$-regular $\tilde \gamma_H \in \tilde H(F)$.  The
sum is finite, running over representatives $\gamma_G$ of regular
semisimple conjugacy classes such that
$\Delta(\tilde\gamma_H,\gamma_G)\ne0$.  The group $\tilde H$ is
obtained as an admissible $z$-extension from $H$ as explained in
\cite[\S 4.4]{LSxf}. It has an associated multiplicative character
$\theta$ on $Z(H)(F)$.

We recall that $(G,H)$ admits {\it local endoscopic $\Delta$-matching}
if for any $f\in C_c^\infty(G(F))$ we can find $f^H \in C_c^\infty(H(F))$
such that
\[
SO( \gamma_H,f^{ H}) 
 = \sum_{\gamma_G} \Delta_{loc}(\gamma_H,\gamma_G)O(\gamma_G,f)
\]
for all strongly $G$-regular elements $\gamma_H$ near $1$ in $H(F)$.
We allow the size of the neighborhood of $1$ to depend on $f^H$.
In fact, it is enough to match finitely many functions that span the
dual to the space of invariant distributions supported on the
unipotent set.

For each semisimple element $\epsilon_H \in H(F)$ that is the image of
some $\epsilon_G\in G(F)$, Langlands and Shelstad construct an
endoscopic pair $(G_{\epsilon_G},H_{\epsilon_H})$ with corresponding
transfer factor $\Delta_\epsilon$.  The group $G_{\epsilon_G}$ is the
connected centralizer of $\epsilon_G$.

We will need the following result from \cite{LSd}*{2.3.A}.

\begin{theorem}[Langlands-Shelstad]  Suppose all pairs
  $(G_{\epsilon_G},H_{\epsilon_H})$ have local endoscopic
  $\Delta_\epsilon$-matching, then $(G,H)$ has endoscopic $\Delta$-matching.
\end{theorem}

Conversely, if endoscopic matching holds on the entire group, then
local endoscopic matching holds in particular at the identity:
$\epsilon_G = 1$ and $\epsilon_H =1$.

Note that the admissible $z$-extensions are needed to formulate the
statement of endoscopic $\Delta$-matching, but they do not appear in
the statement of local endoscopic matching.  Thus, the admissible
$z$-extensions will play no part in our proof.

\begin{proof}[Proof of Theorem~\ref{thm:xfer}] 
Let $G$ be a definable reductive group over a cocycle
  space $Z$.

  Over a nonarchimedean field, there are only a finite number of endoscopic
  groups, up to conjugacy.  Also, there are only finitely many
  different centralizers that are obtained by descent, up to conjugacy
  \cite[\S 2.2]{LSd}.  We can state this uniformly as the field
  varies: there are finitely many definable connected reductive
  centralizers each having finitely many definable reductive
  endoscopic groups such that for all fields of sufficiently large
  residual characteristic, all the centralizers and their endoscopic
  groups are obtained from these finitely many definable groups by
  specialization to the field in question up to conjugation of
  centralizers and equivalence of endoscopic data.  For each separate
  definable pair $(G_{\epsilon_G},H_{\epsilon_H})$, we will obtain a
  natural number $m$ such that the our arguments work for all fields $F\in
  \op{Loc}_m$.  Then the maximum of all such $m$ will work for all
  descent data.

  By Theorem~\ref{thm:local}, for each endoscopic pair
  $(G_{\epsilon_G},H_{\epsilon_H})$, there exists $m$ such that we
  have local endoscopic $\Delta_\epsilon$ matching for all
  $F\in\op{Loc}_m$.  Note that the local endoscopic matching is
  a direct consequence of the endoscopic matching of Shalika germs.  This
  completes the proof.
\end{proof}

\section{Classification of definable reductive groups}
\label{sec:classification}

We rely on the classification of reductive groups from \cite{Gille},
\cite{Tits}, \cite{Sel}, \cite{Petrov}, \cite{reeder2010torsion},
closely following the presentation in \cite{Gross}.  All reductive
groups in this section are assumed to split over a tamely ramified
field extension.  We recall that we do not have access to a
Frobenius element of the Galois group, which is used in essential ways
in the classification.  We obtain a slightly weaker classification
using $\op{qFr}$, a fixed generator of $\Sigma^{unr}$.

We describe the extent to which nonisomorphic reductive groups over a
given field $F$ can have the same set of fixed choices as described in
Section~\ref{sec:fixed}.  For simplicity, assume that $G$ and $G'$ are
forms of a simple, simply connected, split group $G^{**}$ over $F$.
Assume that $G$ and $G'$ have identical fixed choices.  There are
obvious cases when this occurs:
\begin{enumerate}
\item $G$ and $G'$ are inner forms of $SL(n)$, and their invariants in
  the Brauer group have equal denominators, when expressed in lowest terms.
\item $G$ and $G'$ have isomorphic split outer forms, are not split,
  are quasi-split, and split over different ramified extensions of the
  same degree.  For example, we may have two quasi-split special
  unitary groups that split over different ramified quadratic
  extensions.
\item $G$ and $G'$ are not quasi-split, but are inner forms of
  quasi-split groups $G^*$ and $G'^{*}$ described in the previous
  situation (and that split over ramified quadratic extensions).  
  For example, we may have two nontrivial inner forms of
  quasi-split special unitary groups that split over different
  ramified quadratic extensions.
\end{enumerate}

\begin{rem}\label{rem:rem}
In the second case, the group of automorphisms of a connected
    Dynkin diagram, when nontrivial, is cyclic of prime degree except
    in the case of $D_4$, which has automorphism group $S_3$.  From
    this it follows that different tamely ramified extensions of the
    same degree have isomorphic Galois groups and isomorphic inertia
    subgroups.
\end{rem}

\begin{lem} Let $G$ and $G'$ are forms of a simple, simply connected
  split group $G^{**}$ over a nonarchimedean field $F$, for which the
  fixed choices can be chosen to be equal.  Assume the groups split
  over a tamely ramified extension of $F$.  Then $G$ and $G'$ fall
  into one of the three cases just enumerated.
\end{lem}

We remark that the classification in this lemma with respect to
specific fixed data is not known to coincide with the classification
with respect to arbitrary definable data.  This is a question we leave
unanswered.

\begin{proof}
  Specifically, suppose that we are given $F$ and an enumerated Galois
  group $\op{Gal}(K/F)$ that splits $G$.  Suppose that it has a short
  exact sequence
\[
1 \to\op{Gal}(K/E)\to \op{Gal}(K/F) \to \op{Gal}(E/F)\to 1,
\]
where $E$ is the maximal unramified extension of $F$ in $K$,
that is isomorphic with the fixed data
\[
1 \to\Sigma^t\to\Sigma\to\Sigma^{unr}\to 1.
\]
The enumeration fixes a distinguished generator $\op{qFr}$ of
$\Sigma^{unr}$, which need not correspond with the Frobenius generator
of $\op{Gal}(E/F)$.

In what follows, we use the notation of the fixed choices in
Subsection~\ref{sec:fixed}.  We refer to \cite{Gross} for a detailed
treatment of this subject material.

First, consider the case $e=1$, where $e=[\Sigma:\Sigma^t]$.  This
means that the group splits over an unramified extension.  Part of the
fixed data is an action $\phi:\Sigma^{unr}\to \op{Aut}(\R)$ on the
affine diagram.  If the diagram is not $\cal{A}_n$, then $\phi(\op{qFr})$
is determined by its order up to an automorphism of $\R$.  Thus, by
adjusting by an automorphism, we may assume that
$\phi(\op{qFr})=\phi(\op{Fr})$, where $\op{Fr}$ is the Frobenius
automorphism of $\op{Gal}(E/F)$, under an identification of the
Galois group with $\Sigma^{unr}$.  We then know the action of Frobenius 
on $\R$, which is enough to determine the group $G$ up to isomorphism.
Thus the fixed data leaves no flexibility in the isomorphism class of the group.

When $e=1$ and $\R = \cal{A}_n$, the image of $\phi$ determines
a choice of Tits index (or equivalently to the set of simple roots
specifying the minimal parabolic subgroup).  We can specify the
minimal parabolic but not the particular form of the anisotropic
kernel (unless we are given the Frobenius generator).  In other words,
fixing the group $\phi(\Sigma^{unr})$ rather than $\phi(\op{Fr})$ is
sufficient to determine the denominator of the invariant in the Brauer
group, but not the numerator.  This is the first case.

Now consider the case $e>1$.  These are groups that split over a
ramified extension.  The fixed data determines whether the group is
quasi-split.  The quasi-split groups are precisely those for which
$\phi(\op{qFr})=1$ (that is, the case $c=1$ in \cite[\S7]{Gross}).

Assume $e>1$ and the group is quasi-split. 
In this case, the fixed data gives
\[
\rho_G:\Sigma\to \op{Aut}(\op{Dyn}).
\]
The image of $\Sigma$ is cyclic with corresponding extension totally
ramified, except in the case of an $S_3$-action on the Dynkin diagram.
This homomorphism and the identification of the short exact sequence
for $\op{Gal}(K/F)$ with that for $\Sigma$ determine the group up to
isomorphism.  See Remark~\ref{rem:rem}.  This is the second case.

Finally, consider the case of non quasi-split forms with $e>1$.  In
this situation, the only affine diagrams ${}^e\R$ with automorphisms
are ${}^2\cal{A}_{2n+1}$ (with one inner form) and ${}^2\cal{D}_n$
(with one inner form).  There is a unique inner form, so it is
necessarily uniquely determined by the fixed data, once the ramified
quadratic splitting field $K/F$ is given. This is the third case.
\end{proof}

\section{Open problems}\label{sec:open-problems}

\subsection{Igusa theory}

Langlands and Igusa have results about the asymptotic behavior of
integrals \cite{langlands1983orbital} \cite{igusa1978lectures}.  Under
certain assumptions, a family of $p$-adic integrals with parameter $x$
in a nonarchimedean field $F$ can be expanded in a finite asymptotic
series
\begin{equation}\label{eqn:igusa}
\sum_{a,b,\theta} c_{a,b,\theta} \,
\theta(x) \,|x|^a \, (\op{ord}x)^b\, 
\end{equation}
for some integers $a,b$, multiplicative characters $\theta:F^\times\to
\ring{C}^\times$, and for some constants $c_{a,b,\theta}$.  The
constants are given as principal-valued integrals.  The proof of the
asymptotic series relies on the Igusa zeta function.

In the article, we have presented some results about the asymptotic
behavior of integrals.  This suggests that it may be possible to adapt
the asymptotic series expansion (Equation \ref{eqn:igusa}) to a motivic
context.  A general motivic theory of multiplicative characters is
not yet available.  Nonetheless, an inspection of the proofs in
Langlands suggest that it might be possible to obtain an
expansion similar to (Equation \ref{eqn:igusa}) using far less than a general
theory of multiplicative characters.

\subsection{Twisted and nonstandard endoscopy}

In this article, we have transfered the endoscopic matching of smooth
functions for standard endoscopy.  Kottwitz, Shelstad, Waldspurger,
and Ng\^o have developed a corresponding theory for twisted endoscopy
\cite{kottwitz1999foundations}, \cite{waldspurger2008endoscopie},
\cite{ngo2010lemme}.  We expect that these results can be combined
with our methods to obtain transfer results for twisted endoscopic
matching of smooth functions.

\subsection{Smooth transfer and the trace formula}

Our article relies on Waldspurger's proof that the fundamental lemma
implies smooth matching in characteristic zero.  Waldspurger's primary
tool is a global Lie algebra trace formula based on the adelic Poisson
summation formula on the Lie algebra.  Hrushovski-Kazhdan and
Chambert-Loir-Loeser have developed motivic Poisson summation
formulas \cite{hrushovski2009motivic}, \cite{chambert2013motivic}.  In
fact, Hrushovski and Kazhdan have applied the motivic Poisson
summation formula to show that if two local test functions on two division
algebras have matching orbital integrals, then their Fourier
transforms also have matching orbital integrals~\cite[Theorem
1.1]{hrushovski2009motivic}.  Their methods and result are closely
related to Waldspurger's smooth matching result.  This suggests the
problem of adapting Waldspurger's proof to a motivic setting.

\subsection{Definably indistinguishable reductive groups}

In Section~\ref{sec:classification}, we described the extent to which
fixed choices may be used to distinguish various isomorphism classes
of reductive groups.  We may ask further whether {\it groups in the
  same family (1), (2), (3) of Section~\ref{sec:classification} are
  indistinguishable with respect to all (first-order) properties
  expressible in the Denef-Pas language?}

We may also ask {\it to what extent is the harmonic analysis on two groups
in the same family the same?}

We give an example.  Consider the Denef-Pas language with constants in
$\ring{Z}[[\pi]]$, where $\pi$ is simply a variable.  Consider the
quasi-split definable simply connected group of type ${}^eA_n$,
${}^eD_n$, or ${}^eE_6$, with $e=2$ or $3$ as appropriate, associated
with the totally ramified extension $K/\VF$ of degree $e$:
\[
K = \VF[t]/(t^e - \pi),\quad \ord (\pi) = 1.
\]
We make the remarkable observation that the given definable group has
nonisomorphic interpretations as reductive groups over a
nonarchimedean field $F$, simply because nonisomorphic field
extensions of $F$ are obtained by different choices of the uniformizer
$\pi_F\in F$ interpreting $\pi$.  So in particular, for two different
quasi-split unitary groups over $F$, splitting over different ramified
quadratic extensions, the Shalika germs are given by identical
constructible formulas, differing over $F$ only by the specialization
of $\pi$ to different uniformizers $\pi_F\in F$.  This suggests that
the harmonic analyses on two groups in the same family are very closely
related.

\subsection{Definability of nilpotent orbits}

We have carefully avoided the issue of the definability of nilpotent
orbits in this article.  This has required us to use certain roundabout
constructions such as the set $\NF^k_G$.

If we allow a free parameter $N$ that ranges over the nilpotent set,
then the nilpotent orbits are trivially definable by the formula
\[
O(N) = \{N' \mid \exists g\in G,~ \op{ad}g\,(N) = N' \}.
\]  
However, this leaves open the question of whether nilpotent orbits can
be defined in the Denef-Pas language without the use of parameters.
Diwadkar has obtained partial results on this problem
\cite{diwadkar2006nilpoten}.  The results of Barbasch-Moy and DeBacker
on the classification of nilpotent orbits --together with definability
for the Moy-Prasad filtration \cite{CGH} -- reduce this problem to the
definability of nilpotent orbits (or more simply of distinguished
nilpotent orbits) in reductive groups over finite fields.  For some
groups such as $SL(n)$, roots of unity are required to specify the
nilpotent orbits.  But in the adjoint case, the situation is
simpler. This leads to the following question: Let $G=G_{\text{adj}}$
be an adjoint semisimple linear algebraic group defined over a finite
field.  When the characteristic is sufficiently large, is each
nilpotent orbit in the Lie algebra definable in the first-order
language of rings?

\section{Errata}

We mention two corrections to \cite{CHL}.  In Section 3.2 of that
article, the parameter space of unramified extensions requires more
than the single parameter $a$ that was given.  This does not affect
the correctness of the results, but the dimension of the parameter
space needs to be changed throughout the article.

In Section 4.2, the construction of the measure on the stable orbit as
a Leray residue is correct, but the surrounding comments contain
inaccuracies and should be deleted.

\bibliographystyle{alpha}
\bibliography{refs}

\begin{bibdiv}
\begin{biblist}

\bib{adler:anisotropic}{article}{
      author={Adler, Jeffrey~D.},
       title={Refined anisotropic {$K$}-types and supercuspidal
  representations},
        date={1998},
     journal={Pacific J. Math},
      volume={185},
      number={1},
}

\bib{adler-debacker:mktheory}{article}{
      author={Adler, Jeffrey~D.},
      author={DeBacker, Stephen},
       title={{M}urnaghan--{K}irillov theory for supercuspidal representations
  of tame general linear groups, with appendices by {R}eid {H}untsinger and
  {G}opal {P}rasad},
        date={2004},
        ISSN={0075-4102},
     journal={J. Reine Angew. Math.},
      volume={575},
       pages={1\ndash 35},
      review={MR2097545 (2005j:22008)},
}

\bib{adler-debacker:bt-lie}{article}{
      author={Adler, Jeffrey~D.},
      author={DeBacker, Stephen},
       title={Some applications of {B}ruhat--{T}its theory to harmonic analysis
  on the {L}ie algebra of a reductive $p$-adic group},
        date={2002},
        ISSN={0026-2285},
     journal={Michigan Math. J.},
      volume={50},
      number={2},
       pages={263\ndash 286},
      review={MR1914065 (2003g:22016)},
}

\bib{adler-korman}{article}{
      author={Adler, Jeffrey~D.},
      author={Korman, Jonathan},
       title={The local character expansion near a tame, semisimple element},
        date={2007},
     journal={Amer. J. Math.},
      volume={129},
      number={2},
       pages={381\ndash 403},
}

\bib{barbasch-moy}{article}{
      author={Barbasch, Dan},
      author={Moy, Allen},
       title={Local character expansions},
    language={English, with English and French summaries},
        date={1997},
        ISSN={0012-9593},
     journal={Ann. Sci. \'Ecole Norm. Sup. (4)},
      volume={30},
      number={5},
       pages={553\ndash 567},
      review={MR1474804 (99j:22021)},
}

\bib{carter1985finite}{book}{
      author={Carter, Roger~William},
       title={Finite groups of {Lie} type: Conjugacy classes and complex
  characters},
   publisher={Wiley New York},
        date={1985},
      volume={20},
}

\bib{chambert2013motivic}{article}{
      author={Chambert-Loir, Antoine},
      author={Loeser, Fran{\c{c}}ois},
       title={Motivic height zeta functions},
        date={2013},
     journal={arXiv preprint arXiv:1302.2077},
}

\bib{CGH1}{article}{
      author={Cluckers, Raf},
      author={Gordon, Julia},
      author={},
      author={Halupczok, Immanuel},
       title={Integrability of oscillatory functions on local fields: transfer
  principles},
        date={2014},
     journal={Duke Math. J.},
      volume={163},
      number={8},
       pages={1549\ndash 1600},
}

\bib{CGH2}{article}{
      author={Cluckers, Raf},
      author={Gordon, Julia},
      author={},
      author={Halupczok, Immanuel},
       title={Transfer principles for bounds of motivic exponential functions},
        date={2014},
     journal={preprint arXiv:1412.4573},
}

\bib{CGH}{article}{
      author={Cluckers, Raf},
      author={Gordon, Julia},
      author={Halupczok, Immanuel},
       title={Local integrability results in harmonic analysis on reductive
  groups in large positive characteristic},
        date={2014},
     journal={Ann. Sci. \'Ec. Norm. Sup},
      volume={47},
      number={6},
       pages={1163\ndash 1195},
}

\bib{CHL}{incollection}{
      author={Cluckers, Raf},
      author={Hales, Thomas},
      author={Loeser, Fran{\c{c}}ois},
       title={Transfer principle for the fundamental lemma},
        date={2011},
   booktitle={On the stabilization of the trace formula},
   publisher={International Press of Boston},
}

\bib{CL}{article}{
      author={Cluckers, Raf},
      author={Loeser, Fran{\c{c}}ois},
       title={Constructible motivic functions and motivic integration},
        date={2008},
     journal={Inventiones mathematicae},
      volume={173},
      number={1},
       pages={23\ndash 121},
}

\bib{CLe}{article}{
      author={Cluckers, Raf},
      author={Loeser, Fran{\c {c}}ois},
       title={Constructible exponential functions, motivic {Fourier} transform
  and transfer principle},
        date={2010},
     journal={Ann. of Math.},
      volume={171},
      number={2},
       pages={1011\ndash 1065},
}

\bib{debacker:nilp}{article}{
      author={DeBacker, Stephen},
       title={Parametrizing nilpotent orbits via {B}ruhat--{T}its theory},
        date={2002},
        ISSN={0003-486X},
     journal={Ann. of Math.},
      volume={156},
      number={1},
       pages={295\ndash 332},
      review={MR1935848 (2003i:20086)},
}

\bib{diwadkar2006nilpoten}{thesis}{
      author={Diwadkar, Jyotsna~Mainkar},
       title={Nilpotent conjugacy classes of reductive $p$-adic {L}ie algebras
  and definability in {P}as's language},
        type={Ph.D. Thesis},
        date={2006},
}

\bib{Gille}{article}{
      author={Gille, Philippe},
       title={Rationality properties of linear algebraic groups and {Galois}
  cohomology},
        date={2007},
     journal={MCM},
}

\bib{Gross}{article}{
      author={Gross, Benedict~H},
       title={Parahorics},
        date={2012},
        note={preprint},
}

\bib{hrushovski2009motivic}{article}{
      author={Hrushovski, Ehud},
      author={Kazhdan, David},
       title={Motivic {Poisson} summation},
        date={2009},
     journal={Mosc. Math. J},
      volume={9},
      number={3},
       pages={569\ndash 623},
}

\bib{humphreys1975linear}{article}{
      author={Humphreys, James~E},
       title={Linear algebraic groups},
        date={1975},
     journal={New York},
}

\bib{igusa1978lectures}{book}{
      author={Igusa, Jun-ichi},
       title={Forms of higher degree},
      series={Reseach Lectures on Mathematics and Physics},
   publisher={Tata Institute},
        date={1978},
      volume={59},
}

\bib{kim:hecke}{article}{
      author={Kim, Ju-Lee},
       title={Hecke algebras of classical groups over $p$-adic fields and
  supercuspidal representations},
        date={1999},
     journal={Amer. J. Math.},
      volume={121},
       pages={967\ndash 1029},
}

\bib{Kott}{article}{
      author={Kottwitz, Robert},
       title={Transfer factors for {L}ie algebras},
        date={1999},
     journal={Representation Theory of the AMS},
      volume={3},
      number={6},
       pages={127\ndash 138},
}

\bib{kottwitz1999foundations}{article}{
      author={Kottwitz, Robert~E},
      author={Shelstad, Diana},
       title={Foundations of twisted endoscopy},
        date={1999},
     journal={Ast{\'e}risque},
}

\bib{LSd}{incollection}{
      author={Langlands, R},
      author={Shelstad, Diana},
       title={Descent for transfer factors},
        date={1990},
   booktitle={The {Grothendieck Festschrift}},
   publisher={Springer},
       pages={485\ndash 563},
}

\bib{langlands1983orbital}{article}{
      author={Langlands, Robert~P},
       title={Orbital integrals on forms of {SL(3)}, {I}},
        date={1983},
     journal={Amer. J. Math.},
       pages={465\ndash 506},
}

\bib{LSxf}{article}{
      author={Langlands, Robert~P},
      author={Shelstad, Diana},
       title={On the definition of transfer factors},
        date={1987},
     journal={Math. Ann.},
      volume={278},
      number={1},
       pages={219\ndash 271},
}

\bib{mcninch}{article}{
      author={McNinch, George},
       title={Nilpotent orbits over ground fields of good characteristic},
        date={2004},
     journal={Math. Ann.},
      volume={329},
       pages={49\ndash 85},
}

\bib{ngo2010lemme}{article}{
      author={Ng{\^o}, Bao~Ch{\^a}u},
       title={Le lemme fondamental pour les algebres de {Lie}},
        date={2010},
     journal={Publications math{\'e}matiques de l'IH{\'E}S},
      volume={111},
      number={1},
       pages={1\ndash 169},
}

\bib{Petrov}{article}{
      author={Petrov, Victor},
      author={Stavrova, Anastasia},
       title={The {Tits} indices over semilocal rings},
        date={2011},
     journal={Transformation groups},
      volume={16},
      number={1},
       pages={193\ndash 217},
}

\bib{reeder2010torsion}{article}{
      author={Reeder, Mark},
       title={Torsion automorphisms of simple {Lie} algebras},
        date={2010},
     journal={Enseign. Math.(2)},
      volume={56},
      number={1-2},
       pages={3\ndash 47},
}

\bib{Sel}{article}{
      author={Selbach, Martin},
       title={Klassificationstheorie halbeinfacher algebraischer {Gruppen}},
        date={1976},
        note={unpublished},
}

\bib{shalika1972theorem}{article}{
      author={Shalika, Joseph~A},
       title={A theorem on semi-simple p-adic groups},
        date={1972},
     journal={Ann. of Math.},
       pages={226\ndash 242},
}

\bib{Tits}{inproceedings}{
      author={Tits, Jacques},
       title={Classification of algebraic semisimple groups},
        date={1966},
   booktitle={Algebraic groups and discontinuous subgroups ({Proc. Sympos. Pure
  Math.}, {Boulder, Colo.}, 1965)},
      volume={9},
       pages={33\ndash 62},
}

\bib{waldspurger1997lemme}{article}{
      author={Waldspurger, J-L},
       title={Le lemme fondamental implique le transfert},
        date={1997},
     journal={Compositio Mathematica},
      volume={105},
      number={02},
       pages={153\ndash 236},
}

\bib{waldspurger2008endoscopie}{book}{
      author={Waldspurger, Jean-Loup},
       title={L'endoscopie tordue n'est past si tordue},
   publisher={AMS},
        date={2008},
}

\end{biblist}
\end{bibdiv}

\end{document}